\documentclass[review,hidelinks,onefignum,onetabnum]{siamart220329}

\usepackage{lipsum}
\usepackage{amsfonts}
\usepackage{graphicx}
\usepackage{epstopdf}
\usepackage{algorithmic}
\ifpdf
  \DeclareGraphicsExtensions{.eps,.pdf,.png,.jpg}
\else
  \DeclareGraphicsExtensions{.eps}
\fi

\newsiamremark{remark}{Remark}
\newsiamremark{hypothesis}{Hypothesis}
\crefname{hypothesis}{Hypothesis}{Hypotheses}
\newsiamthm{claim}{Claim}

\headers{Additive Schwarz Method For Eigenvalue Problems}{Qigang Liang and Xuejun Xu}

\title {A Two-Level Additive Schwarz Method for Computing Interior Multiple and Clustered Eigenvalues of Symmetric Elliptic Operators
\thanks{Submitted to the editors DATE.}}

\author {Qigang Liang\thanks{School of Mathematical Sciences, Tongji University, Shanghai 200092, China and Key Laboratory of Intelligent Computing and Applications (Tongji University), Ministry of Education (\email{qigang$\_$liang@tongji.edu.cn}, \email{xuxj@tongji.edu.cn}).}
\and Xuejun Xu\footnotemark[2]}

\usepackage{amsopn}

\usepackage{bm}
\usepackage{dcolumn}
\usepackage{stmaryrd}
\usepackage[title]{appendix}
\usepackage{mathrsfs}
\usepackage{tikz}
\usetikzlibrary{decorations.pathreplacing,calligraphy}

\newtheorem{example}[theorem]{Example}

\newtheorem{assumption}{Assumption}
\numberwithin{equation}{section}
\nolinenumbers
\begin{document}

\maketitle

\begin{abstract}
In this paper, we propose an efficient two-level additive Schwarz method for solving large-scale eigenvalue problems arising from the finite element discretization of symmetric elliptic operators, which may compute efficiently more interior multiple and clustered eigenvalues other than only the first several smallest eigenvalues. The proposed method is parallel in two ways: one is to solve the preconditioned Jacobi-Davidson correction equations by the two-level additive Schwarz preconditioner, the other is to solve different clusters of eigenvalues (see Figure \ref{Figure1_targeted_eigenvalues} in Introduction) simultaneously. It only requires computing a series of parallel subproblems and solving a small-dimensional eigenvalue problem per iteration for a cluster of eigenvalues. Based on some new estimates and tools, we provide a rigorous theoretical analysis to prove that convergence factor of the proposed method is bounded by $\gamma=c(H)\rho(\frac{\delta}{H},d_{m}^{-},d_{M}^{+})$, where $H$ is the diameter of subdomains, $\delta$ is the overlapping size and $d_{m}^{-},d_{M}^{+}$ are the distances from both ends of the targeted eigenvalues to others (see Figure \ref{Figure2_targeted_eigenvalues} in Introduction). The positive number $\rho(\frac{\delta}{H},d_{m}^{-},d_{M}^{+})<1$ is independent of the fine mesh size and the internal gaps among the targeted eigenvalues. The $H$-dependent constant $c(H)$ decreases monotonically to 1, as $H\to 0$, which means the more subdomains lead to the better convergence. Numerical results supporting our theory are given.
\end{abstract}


\section{Introduction}\label{sec1}
\par Solving large-scale eigenvalue problems arising from the finite element discretization of partial differential operators is one of the fundamental problems in modern science and engineering. The problem is extremely significant, so finite element methods for PDE eigenvalue problems and efficient eigensolvers for corresponding large-scale discrete systems have been extensively studied in the literature (see, e.g., \cite{MR962210,MR2652780,MR3647956,MR4136540,MR3407250,MR4948586,MR2430983,MR3347459,MR2206452,MR4735245,MR3266953,MR4474059,MR2214750}).
However, unlike computing a simple eigenvalue of PDE operators, the design and analysis of large-scale eigensolvers for computing interior multiple and clustered eigenvalues of PDE operators (see Figure \ref{Figure1_targeted_eigenvalues}), pose a greater challenge. 

\begin{figure}[H]
\centering
\begin{tikzpicture}[scale=1.2]
\draw[->](-3,0)--(7,0);
\draw (-2,0) --(-2,0.2) node[below=5pt,scale=1]{$\lambda_{m_{1}}$};
\draw (-1.8,0) --(-1.8,0.1) node[below=3pt,scale=1]{\ };
\draw (-1.75,0) --(-1.75,0.1) node[below=3pt,scale=1]{\ };
\draw (-1.7,0) --(-1.7,0.1) node[below=3pt,scale=1]{\ };
\draw (-1.6,0) --(-1.6,0.1) node[below=3pt,scale=1]{\ };
\draw (-1.65,0) --(-1.65,0.1) node[below=3pt,scale=1]{\ };
\draw (-1.625,0) --(-1.625,0.1) node[below=3pt,scale=1]{\ };
\draw (-1.5,0) --(-1.5,0.1) node[below=3.6pt,scale=1]{\ };
\draw (-1.3,0) --(-1.3,0.1) node[below=3.6pt,scale=1]{\ };
\draw (-1.1,0) --(-1.1,0.1) node[below=3.6pt,scale=1]{\ };
\draw (-1,0) --(-1,0.2) node[below=5pt,scale=1]{$\lambda_{M_{1}}$};
\draw [decorate, decoration
       = {calligraphic brace,
            raise=5pt,     
            aspect=0.5,    
            amplitude=3pt  
         }
       ] (-2,0.1) --  (-1,0.1);
\draw (-1.5,0.2) -- (-1.5,0.2) node[above=3pt,scale=1]{cluster 1};
\draw (0.2,0) --(0.2,0.2) node[below=5pt,scale=1]{$\lambda_{m_{2}}$};
\draw (0.3,0) --(0.3,0.1) node[below=3pt,scale=1]{\ };
\draw (0.35,0) --(0.35,0.1) node[below=3pt,scale=1]{\ };
\draw (0.4,0) --(0.4,0.1) node[below=3pt,scale=1]{\ };
\draw (0.5,0) --(0.5,0.1) node[below=3pt,scale=1]{\ };
\draw (0.6,0) --(0.6,0.1) node[below=3pt,scale=1]{\ };
\draw (0.65,0) --(0.65,0.1) node[below=3pt,scale=1]{\ };
\draw (0.7,0) --(0.7,0.1) node[below=3pt,scale=1]{\ };
\draw (0.8,0) --(0.8,0.1) node[below=3pt,scale=1]{\ };
\draw (0.95,0) --(0.95,0.1) node[below=3.6pt,scale=1]{\ };
\draw (0.975,0) --(0.975,0.1) node[below=3.6pt,scale=1]{\ };
\draw (0.9,0) --(0.9,0.1) node[below=3.6pt,scale=1]{\ };
\draw (1,0) --(1,0.2) node[below=5pt,scale=1]{$\lambda_{M_{2}}$};
\draw [decorate, decoration
       = {calligraphic brace,
            raise=5pt,     
            aspect=0.5,    
            amplitude=3pt  
         }
       ] (0.2,0.1) --  (1,0.1);
\draw (0.6,0.2) -- (0.6,0.2) node[above=3pt,scale=1]{cluster 2};
\draw (2.2,0) --(2.2,0.2) node[below=5pt,scale=1]{$\lambda_{m_{3}}$};
\draw (2.3,0) --(2.3,0.1) node[below=3pt,scale=1]{\ };
\draw (2.33,0) --(2.33,0.1) node[below=3pt,scale=1]{\ };
\draw (2.35,0) --(2.35,0.1) node[below=3pt,scale=1]{\ };
\draw (2.38,0) --(2.38,0.1) node[below=3pt,scale=1]{\ };
\draw (2.4,0) --(2.4,0.1) node[below=3pt,scale=1]{\ };
\draw (2.5,0) --(2.5,0.1) node[below=3pt,scale=1]{\ };
\draw (2.6,0) --(2.6,0.1) node[below=3pt,scale=1]{\ };
\draw (2.65,0) --(2.65,0.1) node[below=3pt,scale=1]{\ };
\draw (2.7,0) --(2.7,0.1) node[below=3pt,scale=1]{\ };
\draw (2.8,0) --(2.8,0.1) node[below=3pt,scale=1]{\ };
\draw (2.9,0) --(2.9,0.1) node[below=3.6pt,scale=1]{\ };
\draw (3,0) --(3,0.2) node[below=5pt,scale=1]{$\lambda_{M_{3}}$};
\draw [decorate, decoration
       = {calligraphic brace,
            raise=5pt,     
            aspect=0.5,    
            amplitude=3pt  
         }
       ] (2.2,0.1) --  (3,0.1);
\draw (2.6,0.2) -- (2.6,0.2) node[above=3pt,scale=1]{cluster 3};
\draw (4,0) --  (4,0) node[above=3pt,scale=1]{$......$};
\draw (5.2,0) --(5.2,0.2) node[below=5pt,scale=1]{$\lambda_{m_{j}}$};
\draw (5.3,0) --(5.3,0.1) node[below=3pt,scale=1]{\ };
\draw (5.35,0) --(5.35,0.1) node[below=3pt,scale=1]{\ };
\draw (5.4,0) --(5.4,0.1) node[below=3pt,scale=1]{\ };
\draw (5.5,0) --(5.5,0.1) node[below=3pt,scale=1]{\ };
\draw (5.6,0) --(5.6,0.1) node[below=3pt,scale=1]{\ };
\draw (5.65,0) --(5.65,0.1) node[below=3pt,scale=1]{\ };
\draw (5.67,0) --(5.67,0.1) node[below=3pt,scale=1]{\ };
\draw (5.675,0) --(5.675,0.1) node[below=3pt,scale=1]{\ };
\draw (5.7,0) --(5.7,0.1) node[below=3pt,scale=1]{\ };
\draw (5.8,0) --(5.8,0.1) node[below=3pt,scale=1]{\ };
\draw (5.9,0) --(5.9,0.1) node[below=3.6pt,scale=1]{\ };
\draw (6,0) --(6,0.2) node[below=5pt,scale=1]{$\lambda_{M_{j}}$};
\draw [decorate, decoration
       = {calligraphic brace,
            raise=5pt,     
            aspect=0.5,    
            amplitude=3pt  
         }
       ] (5.2,0.1) --  (6,0.1);
\draw (5.6,0.2) -- (5.6,0.2) node[above=3pt,scale=1]{cluster $j$};
\draw (7,0) -- (7,0) node[right=3pt,scale=1]{$\mathbb{R}$};
\end{tikzpicture}
\caption{Different clusters of eigenvalues: $\lambda_{m_{j}}\leq \lambda_{m_{j}+1}\leq \cdots\leq \lambda_{M_{j}},\ j=1,2,\cdots $}
\label{Figure1_targeted_eigenvalues}
\end{figure}
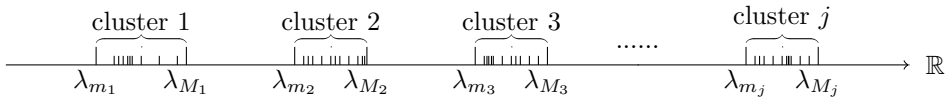

\par Regarding the interior multiple and cluster eigenvalues of partial differential operators, many scholars study their a priori and a posteriori error estimates, and provide numerous new mathematical tools for rapidly solving and analyzing interior multiple and cluster eigenvalues. Babu\v{s}ka and Osborn \cite{MR962210} take advantage of finite element methods to approximate the interior multiple eigenvalues and corresponding eigenfunctions of symmetric elliptic operators. More than ten years later, Knyazev and Osborn in \cite{MR2206452} give a new a priori error estimate for multiple and clustered eigenvalues of symmetric elliptic eigenvalue problems. It is shown in \cite{MR2206452} that the a priori error for eigenvalues in a cluster is independent of interior gaps among targeted eigenvalues in this cluster $\{\lambda_{i}\}_{i=m}^{M}$, depending on the distance $\min\{d_{m}^{-},d_{M}^{+}\}$ from both ends of the targeted eigenvalues to others (see Figure \ref{Figure2_targeted_eigenvalues}). For a posteriori error, Dai et al. \cite{MR3407250} develop an a posteriori error estimator for multiple eigenvalues, and prove the convergence and quasi-optimal complexity of the adaptive finite element methods (AFEM). Subsequently, Gallistl \cite{MR3347459} studies clustered eigenvalues, and proves the convergence and quasi-optimal complexity of AFEM. This idea has been further extended to the higher-order AFEM \cite{MR3532806}, the non-conforming AFEM \cite{MR3259027}, and the mixed AFEM \cite{MR3647956}. Lately, Canc\`es et al. \cite{MR4136540} present a guaranteed a posteriori error estimator for interior multiple and clustered eigenvalues (from $m$th to $M$th, see Figure 2) of symmetric elliptic operators. They give a rigorous analysis for the reliability and efficiency of the a posteriori error estimator. The constants in the proof are independent of the interior gaps among the targeted eigenvalues $\{\lambda_{i}\}_{i=m}^{M}$, depending on the distance $\min\{d_{m}^{-},d_{M}^{+}\}$ from both ends of the targeted eigenvalues to others (see Figure \ref{Figure2_targeted_eigenvalues}). More recently, based on the Courant-Fischer principle and the Davis–Kahan method, Liu and Vejchodsk\'{y}  \cite{MR4474059} propose two methods to provide greatly improved accuracy of fully computable a posteriori error estimate for eigenfunction approximations corresponding to interior multiple and clustered eigenvalues, which make AFEMs solve interior multiple and clustered eigenvalues efficiently. 

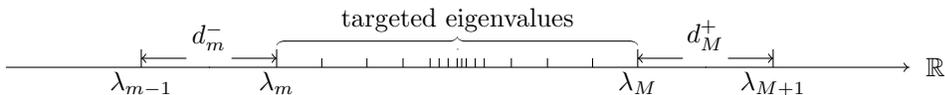
\begin{figure}[H]
\centering
\begin{tikzpicture}[scale=1.2]
\draw[->](-5,0)--(5,0);
\draw (-3.5,0) --(-3.5,0.2) node[below=5pt,scale=1]{$\lambda_{m-1}$};
\draw (-2,0) --(-2,0.2) node[below=5pt,scale=1]{$\lambda_{m}$};
\draw (-1.5,0) --(-1.5,0.1) node[below=3pt,scale=1]{\ };
\draw (-1,0) --(-1,0.1) node[below=3.6pt,scale=1]{\ };
\draw (-0.6,0) --(-0.6,0.1) node[below=3.6pt,scale=1]{\ };
\draw (-0.3,0) --(-0.3,0.1) node[below=3.6pt,scale=1]{\ };
\draw (-0.2,0) --(-0.2,0.1) node[below=3.6pt,scale=1]{\ };
\draw (-0.1,0) --(-0.1,0.1) node[below=3.6pt,scale=1]{\ };
\draw (0,0) -- (0,0.1) node[below=3.6pt,scale=1]{\ };
\draw (0.05,0) --(0.05,0.1) node[below=3.6pt,scale=1]{\ };
\draw (0.1,0) --(0.1,0.1) node[below=3.6pt,scale=1]{\ };
\draw (0.2,0) --(0.2,0.1) node[below=3.6pt,scale=1]{\ };
\draw (0.3,0) --(0.3,0.1) node[below=3.6pt,scale=1]{\ };
\draw (0.6,0) --(0.6,0.1) node[below=3.6pt,scale=1]{\ };
\draw (1,0) --(1,0.1) node[below=3.6pt,scale=1]{\ };
\draw (1.5,0) --(1.5,0.1) node[below=3.6pt,scale=1]{\ };
\draw (2,0) --(2,0.2) node[below=5pt,scale=1]{$\lambda_{M}$};
\draw (3.5,0) --(3.5,0.2) node[below=5pt,scale=1]{$\lambda_{M+1}$};
\draw [decorate, decoration
       = {calligraphic brace,
            raise=5pt,     
            aspect=0.5,    
            amplitude=3pt  
         }
       ] (-2,0.1) --  (2,0.1);
\draw (0,0.2) -- (0,0.2) node[above=3pt,scale=1]{targeted eigenvalues};
\draw[->] (-2.5,0.1) -- (-2,0.1);
\draw[<-] (-3.5,0.1) -- (-3,0.1);
\draw (-2.75,0) --  (-2.75,0) node[above=3pt,scale=1]{$d_{m}^{-}$};
\draw[->] (3,0.1) -- (3.5,0.1);
\draw[<-] (2,0.1) --  (2.5,0.1);
\draw (2.75,0) --  (2.75,0) node[above=3pt,scale=1]{$d_{M}^{+}$};
\draw (5,0) -- (5,0) node[right=3pt,scale=1]{$\mathbb{R}$};
\end{tikzpicture}
\caption{targeted eigenvalues: $\lambda_{m}\leq \lambda_{m+1}\leq \cdots\leq \lambda_{M}$}
\label{Figure2_targeted_eigenvalues}
\end{figure}

\par For solving elliptic eigenvalue problems, two-grid methods have been adopted, and achieve asymptotic optimal accuracy under the conditions that $h = O(H^{i})$ respectively ($i = 2$ or $4$), as evidenced in \cite{MR1677419,MR2785459,MR2831063}. Two-grid methods are also used to solve interior multiple eigenvalues of elliptic operators. It is also worth noting the insightful approach developed by Xu and Zhou \cite{MR2008551}, which utilizes a coarse space together with a sequence of local parallel subspaces to handle low- and high-frequency error components. A similar strategy is adopted in the present work to efficiently treat interior multiple and clustered eigenvalues. For an algebraic eigenvalue problem, many types of block methods have been designed and analyzed for computing extreme or interior multiple and clustered eigenvalues of a linear operator arising from finite element (difference) approximation of PDE operators (see, e.g., \cite{MR3396212,MR3383329,MR3635864,MR3434007}), among which a block type of the Jacobi-Davidson method proposed in \cite{MR1778354} is one of the most popular methods in practice. For large-scale discrete PDE eigenvalue problems, recently, we presented an efficient two-level block preconditioned Jacobi-Davidson (BPJD) method (see \cite{MR4735245}) for solving the first several smallest eigenvalues of symmetric elliptic operators. In many actual applications, it is necessary to compute a relatively large number of interior multiple and clustered eigenvalues with the corresponding eigenfunctions of elliptic operators. However, when the number of eigenpairs is relatively large, the two-level BPJD method in \cite{MR4735245} requires a relatively large amount of memory to store the higher dimensional trial subspace which is used to solve the Rayleigh-Ritz problem. It may be ineffective for solving more interior multiple and clustered eigenvalues. Therefore, designing large-scale efficient solvers for computing directly interior multiple and clustered eigenvalues with the corresponding eigenfunctions is a significant task and will exceedingly benefit the high-performance computation.

\par To address the above issue, we propose an efficient two-level additive Schwarz method for solving interior multiple and clustered eigenvalues with the corresponding eigenfunctions of symmetric elliptic operators. On every outer iteration, the proposed method only needs to solve several preconditioned Jacobi-Davidson correction equations corresponding to the targeted eigenvalues, and then take the iterative approximations of targeted eigenpairs by solving a small-dimensional eigenvalue problem in a trial subspace. Compared to the two-level BPJD method proposed in \cite{MR4735245}, it is easier to compute more interior multiple and clustered eigenpairs because the interior multiple and clustered eigenvalues (from $m$th to $M$th eigenvalues, see Figure \ref{Figure2_targeted_eigenvalues}) can be computed  directly, without computing the first $m-1$ smallest eigenvalues on fine mesh. It is different from traditional preconditioned subspace iteration methods such as Local Optimal Block Preconditioned Conjugate Gradient (LOBPCG) method (see \cite{MR1861263}) and Block Preconditioned Gradient Descent (BPGD) method (see \cite{MR3985474} and references therein).  

\par We must emphasize that it is nontrivial to give a rigorous theoretical analysis for the two-level additive Schwarz method for computing interior multiple and clustered eigenvalues. Firstly, many important results in \cite{MR4735245,MR3749389,MR3957894,MR4628762} developed for the first or the first several eigenpairs of symmetric elliptic operators are not applicable. Secondly, it is essential to ensure that the constants in all inequalities are independent of the gaps among the targeted interior eigenvalues. With the aid of a combination of some mathematical tools, including analyzing technique of high and low frequency errors, constructions of two auxiliary eigenvalue problems and the analysis framework of two-level additive Schwarz method, we give a rigorous analysis for the proposed method, and then obtain the following convergence factor: 
\begin{align}\label{align_introduction_1}
\sum_{i=m}^{M}(\lambda_{i}^{k+1}-\lambda_{i}^{h})\leq \gamma \sum_{i=m}^{M}(\lambda_{i}^{k}-\lambda_{i}^{h}),
\end{align}
where $\gamma=c(H)\rho(\frac{\delta}{H},d_{m}^{-},d_{M}^{+})$. The positive number $\rho(\frac{\delta}{H},d_{m}^{-},d_{M}^{+})<1$ is independent of the fine mesh size $h$ and the internal gaps among the targeted eigenvalue, depending on the distance from both ends of the targeted eigenvalues to others and the ratio $\frac{\delta}{H}$. The $H$-dependent constant $c(H)$ decreases monotonically to 1, as $H\to 0$. Moreover, we have no assumption on the relationship between $H$ and $h$, as well as internal gaps among the targeted interior eigenvalues.

\par The rest of this paper is organized as follows: Some preliminaries are introduced in Section 2. In Section 3,
we present the two-level additive Schwarz method for symmetric elliptic eigenvalue problems. The main convergence analysis is given in Section 4. Finally, we present our numerical results in Section 5 and the conclusion in Section 6.

\section{Model problems and preliminaries}\label{sec2}
\par In this section, we introduce some notations and preliminaries. Throughout this paper, we use the standard notations for the Sobolev spaces $H^{s}(\Omega)$ and $H_{0}^{s}(\Omega)$ with their associated norms and semi-norms, where $s=1,2$ and we assume that $\Omega\subset \mathbb{R}^{d}$ is a convex polygonal domain $(d=2)$ or polyhedral domain $d=3$. 
\par We consider the following eigenvalue problems
\begin{equation}\label{Laplace}
-\Delta u=\lambda u\ \ \ \ x\in \Omega,\ \ \ \ u=0\ \ \ \ \ x\in \partial \Omega,
\end{equation}
and 
\begin{equation}\label{biharmonic}
\Delta^{2} u=\lambda u\ \ \ \ x\in \Omega,\ \ \frac{\partial{u}}{\partial{\bm{n}}}=u=0\ \ \ \ \ x\in \partial \Omega,
\end{equation}
where $\bm{n}$ is the unit outward vector on $\partial \Omega$. For problems \eqref{Laplace} and \eqref{biharmonic}, their variational forms are: find $(\lambda,u)\in \mathbb{R}\times V$ such that
\begin{align}\label{Align_vari_eigen}
a(u,v)=\lambda b(u,v)\ \ \ \ \forall\ v\in V,
\end{align}
where 
$ a(v,w)=\int_{\Omega}\nabla{v} \cdot \nabla{w}\ dx$ for all $v,w \in V:=H_{0}^{1}(\Omega)$ or  
$a(v,w)=\int_{\Omega}\Delta{v} \Delta{w}\ dx$ for all $v,w \in V:=H_{0}^{2}(\Omega)$,
and $b(v,w)=\int_{\Omega}vw dx$\ \ for all $v,w\in L^{2}(\Omega)$. Define the energy norm $||v||_{a}^{2}:=a(v,v)$ for all $v\in V$ and $||w||_{b}^{2}:=b(w,w)$ for all $w\in L^{2}(\Omega)$. For convenience, we define $Rq(v):=\frac{a(v,v)}{b(v,v)}>0$ and $Rt(v)=\frac{1}{Rq(v)}$ for all $v\ (\ne 0) \in V$. For $2s$th ($s=1, 2$) order symmetric elliptic eigenvalue problems with homogeneous essential boundary condition, we set $V=H_{0}^{s}(\Omega)$. Then, their variational forms are similar to \eqref{Align_vari_eigen}. Define $T:L^{2}(\Omega)\to V$ such that for any $f\in L^{2}(\Omega)$, $a(Tf,v)=b(f,v)$ for all $v\in V.$  Since $V$ is embedded compactly in $L^{2}(\Omega)$, we know that $T:L^{2}(\Omega)\to L^{2}(\Omega)$ and $T|_{V}:V\to V$ are symmetric, positive definite and compact operators. By the Hilbert-Schmidt theorem, we know that $Tu_{i}=\lambda_{i}^{-1}u_{i}$ with $0<\lambda_{1}\leq\lambda_{2}\leq \cdots \leq \lambda_{n}\to +\infty$, as $n\to +\infty,$
and the corresponding eigenfunctions satisfying
\begin{align}\label{Align_eigen_ortho}
a(u_{i},u_{j})=\lambda_{i}b(u_{i},u_{j})=\lambda_{i}\delta_{ij},
\end{align}
where $\delta_{ij}$ is the Kronecker delta. In $\{\lambda_{i}\}_{i=1}^{+\infty}$, $\lambda_{i}$ is repeated according to its geometric multiplicity. In this paper, we define $d_{i}^{+}:=\lambda_{i+1}-\lambda_{i}$ and $d_{i}^{-}:=\lambda_{i}-\lambda_{i-1}$, $i=1,2,3,....$, here let $\lambda_{0}=0$. In the following, we are interested in several interior multiple and clustered eigenvalues 
$\{\lambda_{m},\lambda_{m+1},\cdots,\lambda_{M}\}$ and the corresponding eigenfunctions $\{u_{m},u_{m+1},\cdots,u_{M}\}$, where $m>1$. Furthermore, by \eqref{Align_eigen_ortho}, it is easy to know that $V=U_{L}\oplus U_{J}\oplus U_{R}$,
where $U_{L}={\rm span}\{u_{i}\}_{i=1}^{m-1}$, $U_{J}={\rm span}\{u_{i}\}_{i=m}^{M}$, $U_{R}={\rm span}\{u_{i}\}_{i=M+1}^{+\infty}$ and the notation $\oplus$ denotes the orthogonal direct sum of subspaces with respect to $b(\cdot,\cdot)$ (also $a(\cdot,\cdot)$). Now we give a reasonable assumption.
\begin{assumption}\label{Assumption1}
Assume that $\lambda_{m-1}<\lambda_{m}$ and $\lambda_{M}<\lambda_{M+1}$.
\end{assumption}

\begin{remark}
The Assumption \ref{Assumption1} is reasonable. In practical computation, if Assumption \ref{Assumption1} is not met, we may consider to compute more eigenvalues (from $m-q_{1}$ to $M+q_{2}$) so that $\lambda_{m-q_{1}}$ and $\lambda_{M+q_{2}}$ satisfy Assumption \ref{Assumption1}, where $q_{1}$ and $q_{2}$ are positive integers. So, this assumption is not a limitation for our practical computations.
\end{remark}

\par Subsequently, we consider the finite element methods for solving symmetric elliptic eigenvalue problems. Let $V^{h}$ be a conforming finite element space based on a shape regular and quasi-uniform triangular or rectangular partition $\mathcal{T}_{h}$ with the mesh size $h$. Specifically, for the second order elliptic eigenvalue problem, we choose $P_{1}$-conforming finite element space. For the fourth order elliptic eigenvalue problem, we choose the Bogner-Fox-Schmit (BFS) finite element space.  The weak form of finite element discretization of \eqref{Align_vari_eigen} reads: find $(\lambda^{h},u^{h})\in \mathbb{R}\times V^{h}$ such that
\begin{align}\label{Align_vari_eigen_FEM}
a(u^{h},v^{h})=\lambda^{h} b(u^{h},v^{h})\ \ \ \ \forall\ v^{h}\in V^{h}.
\end{align}
Define $T^{h}:L^{2}(\Omega)\to V^{h}$ such that for any $f\in L^{2}(\Omega)$, $a(T^{h}f,v^{h})=b(f,v^{h})$ for all $v^{h}\in V^{h}.$ It is easy to know that $T^{h}|_{V^{h}}:V^{h}\to V^{h}$ is symmetric and positive definite. For convenience of notations, we use $T^{h}$ to represent $T^{h}|_{V^{h}}$ in the following. Moreover, $T^{h}u_{i}^{h}=(\lambda_{i}^{h})^{-1}u_{i}^{h}$ with all discrete eigenvalues satisfying $0<\lambda_{1}^{h}\leq \lambda_{2}^{h}\leq \cdots \leq \lambda_{N_{h}}^{h},$
and the corresponding eigenfunctions satisfying
\begin{align}\label{Align_discre_eigen_ortho}
a(u_{i}^{h},u_{j}^{h})=\lambda_{i}^{h}b(u_{i}^{h},u_{j}^{h})=\lambda_{i}^{h}\delta_{ij},
\end{align}
where $N_{h}=\dim(V^{h})$. Define $A^{h}:V^{h}\to V^{h}$ such that $b(A^{h}w^{h},v^{h})=a(w^{h},v^{h})$ for all $ w^{h},\ v^{h}\in V^{h}$. Then,
$A^{h}u_{i}^{h}=\lambda_{i}^{h}u_{i}^{h}$.
By \eqref{Align_discre_eigen_ortho}, we know
\begin{align}\label{Align_fine_decom}
V^{h}=U_{L}^{h}\oplus U_{J}^{h}\oplus U_{R}^{h},
\end{align}
where $U_{L}^{h}={\rm span}\{u_{i}^{h}\}_{i=1}^{m-1}$, $U_{J}^{h}={\rm span}\{u_{i}^{h}\}_{i=m}^{M}$ and $U_{R}^{h}={\rm span}\{u_{i}^{h}\}_{i=M+1}^{N_{h}}$. For convenience of theoretical analysis, we denote by $Q_{L}^{h}:V^{h}\to U_{L}^{h}$, $Q_{J}^{h}:V^{h}\to U_{J}^{h}$ and $Q_{R}^{h}:V^{h}\to U_{R}^{h}$ three orthogonal projectors with respect to $b(\cdot,\cdot)$ (also $a(\cdot,\cdot)$). In this paper, we call $U_{L}^{h} \ (U_{R}^{h})$ the low (high) frequency subspace.

\par In what follows, we set up the overlapping domain decomposition. Let $\{\Omega_{l}\}_{l=1}^{N}$ be a coarse shape regular and quasi-uniform partition of $\Omega$, and we denote it by $\mathcal{T}_{H}$. We define $H:=\max\{H_{l}\}_{l=1}^{N},\ \text{where}\ H_{l}=\text{diam}(\Omega_{l}).$
The fine shape regular and quasi-uniform partition $\mathcal{T}_{h}$ is obtained by subdividing $\mathcal{T}_{H}$. To obtain overlapping subdomains $\{\Omega_{l}^{'}\}_{l=1}^{N}$, $\Omega_{l}$ are enlarged by adding fine elements inside $\Omega$ layer by layer such that $\partial \Omega_{l}^{'}$ does not cut through any fine element. To measure the overlapping width between neighboring subdomains, we define
$\delta:=\min\{\delta_{l}\}_{l=1}^{N},\ \text{where}\ \delta_{l}={\rm dist} (\partial\Omega_{l}\setminus\partial\Omega,\partial\Omega_{l}^{'}\setminus\partial\Omega).$
Hence, we know that ${\rm diam}(\Omega_{l}^{'})=\mathcal{O}(H_{l})$. The local subspaces may be defined by $V^{(l)}:=V^{h}\cap H_{0}^{s}(\Omega_{l}^{'}),\ s=1,2$. We introduce the finite covering (see \cite{MR2104179}) assumption as follows:

\begin{assumption}\label{Assumption2}
The partition $\{\Omega_{l}^{'}\}_{l=1}^{N}$ may be colored using at most $N_{0}$ colors, in such a way that subdomains with the same color are disjoint. The integer $N_{0}$ is independent of $N$.
\end{assumption}

\par Next, we present a priori error estimates for interior multiple and clustered eigenvalues of symmetric elliptic operators (please see Theorems 2.11 and 3.3 in \cite{MR2206452}). For a priori error estimates of the eigenfunctions, we give two estimates, which is obtained by modified the center and radius of the circle in Appendix of \cite{MR4735245}. Unless otherwise stated, the letters $C$ in this paper denote generic positive constants independent of the fine mesh size $h$ and the internal gaps among the eigenvalues 
$\{\lambda_{i}\}_{i=m}^{M}$, which may be different at different occurrences.

\begin{theorem}\label{theorem_priori}
For problems \eqref{Align_vari_eigen}, if Assumption \ref{Assumption1} holds, then the discrete eigenvalues in \eqref{Align_vari_eigen_FEM}   $\lambda_{m}^{h},\lambda_{m+1}^{h},\cdots\lambda_{M}^{h}$ converge to exact eigenvalues $\lambda_{m},\lambda_{m+1},\cdots \lambda_{M}$ in \eqref{Align_vari_eigen}, respectively, as $h\to 0$. Moreover, there exists a $h_{0}\ (<1)$ such that $0<h<h_{0}$, we have
\begin{align}\label{Align_a_priori_eigenvalue}
\lambda_{i}^{h}-\lambda_{i}\leq Ch^{2},\ \ \ \ i=m,m+1,\cdots,M,
\end{align}
and
\begin{align}\label{Align_a_priori_eigenfun}
\theta_{a}(U_{J},U_{J}^{h})\leq Ch, \ \ \ \ \theta_{b}(U_{J},U_{J}^{h})\leq Ch^{2},
\end{align}
where the constant $C$ depends on $d_{m}^{-}$ and $d_{M}^{+}$, and $\theta_{a}(U_{J},U_{J}^{h}) $ and $\theta_{b}(U_{J},U_{J}^{h})$ denote the gaps between $U_{J}$ and $U_{J}^{h}$ with respect to $||\cdot||_{a}$ and $||\cdot||_{b}$.
\end{theorem}

\section{The two-level additive Schwarz method}\label{sec3}
\par In this section, we present the two-level additive Schwarz method for computing interior multiple and clustered eigenpairs. 
\par In order to present our preconditioner, define $A^{H}:V^{H}\to V^{H}$ such that for any $w^{H}\in V^{H}$, $ b(A^{H}w^{H},v^{H})=a(w^{H},v^{H})$ for all $v^{H}\in V^{H}$, and denote by $Q^{H}:V^{h}\to V^{H}$ a $b(\cdot,\cdot)$-orthogonal projector. Similar to \eqref{Align_fine_decom}, $V^{H}=U_{L}^{H}\oplus U_{J}^{H}\oplus U_{R}^{H}$ holds,
where $U_{L}^{H}={\rm span}\{u_{i}^{H}\}_{i=1}^{m-1}$, $U_{J}^{H}={\rm span}\{u_{i}^{H}\}_{i=m}^{M}$ and $U_{R}^{H}={\rm span}\{u_{i}^{H}\}_{i=M+1}^{N_{H}}$ with $u_{i}^{H}$ being the eigenfunction of $A^{H}$ and $N_{H}=\dim{V^{H}}$. Let $Q_{L}^{H}:V^{H}\to U_{L}^{H}$, $Q_{J}^{H}:V^{H}\to U_{J}^{H}$ and $Q_{R}^{H}:V^{H}\to U_{R}^{H}$ three $b(\cdot,\cdot)$-orthogonal projectors. We also define $A^{(l)}:V^{(l)}\to V^{(l)}\ (l=1,2,\cdots,N)$ such that $b(A^{(l)}w^{(l)},v^{(l)})=a(w^{(l)},v^{(l)})$ for all $v^{(l)}\in V^{(l)}$, and denote by $Q^{(l)}:V^{(l)}\to V^{(l)}$ a $b(\cdot,\cdot)$-orthogonal projector. 
For any $U\subset V^{h}$, $U^{\perp}$ represents the $b(\cdot,\cdot)$-orthogonal complement of $U$. For $i=m,m+1,\cdots, M$, our parallel preconditioner defined as 
\begin{equation}\label{Equation_Preconditioner}
(B_{i}^{k})^{-1}:=(B_{0,i}^{k})^{-1}Q_{R}^{H}Q^{H}+\sum_{l=1}^{N}(B_{l,i}^{k})^{-1}Q^{(l)},
\end{equation}
is to precondition the block Jacobi-Davidson correction equations
\begin{align*}
b((A^{h}-\lambda_{i}^{k})\hat{t}_{i}^{k},v)=b(r_{i}^{k},v)\ \ \ \ \forall\ v\in  (U_{J}^{k})^{\perp},
\end{align*}
as $h\to 0$, where $B_{0,i}^{k}:=A^{H}-\lambda_{i}^{k}$ and $B_{l,i}^{k}:=A^{(l)}-\lambda_{i}^{k}$, $r_{i}^{k}=\lambda_{i}^{k}u_{i}^{k}-A^{h}u_{i}^{k}$, $U^{k}_{J}={\rm span}\{u_{i}^{k}\}_{i=m}^{M}$ with  $(\lambda_{i}^{k}, u_{i}^{k})$ being the current iterative approximation of $(\lambda_{i}^{h}, u_{i}^{h})$, and $\hat{t}_{i}^{k}\in  (U_{J}^{k})^{\perp}$ is the exact solution of correction equations.  Combining the scaling argument, the Poincar\'e inequality and the inverse estimate,  for $s=1,2$ we have 
\begin{align}
\label{align_min_max_eigen}
\begin{aligned}
\lambda_{\min}(B_{0,i}^{k}|_{U_{R}^{H}})&=\lambda_{M+1}^{H}-\lambda_{i}^{k},\ \ \ \ \ \lambda_{\max}(B_{0,i}^{k}|_{U_{R}^{H}})=\mathcal{O}(H^{-2s}),\\
\lambda_{\min}(B_{l,i}^{k})&=\mathcal{O}(H_{l}^{-2s}),\ \ \ \ \ \ \ \ \ \ \ \ \ \lambda_{\max}(B_{l,i}^{k})=\mathcal{O}(h^{-2s}).
\end{aligned}
\end{align}

\begin{table}[H]\small
\begin{tabular}{p{12.6cm}}
\hline
\hline
\textbf{Algorithm 3.1} The two-level additive Schwarz method for interior eigenvalues\\
\hline
$\bm{Step\ 1}$. Given the innitial approximations $(\lambda_{i}^{0},u_{i}^{0})$ of $(\lambda_{i},u_{i})$ such that $\lambda_{i}^{0}<\lambda_{i}^{H}$ and
$$a(u_{i}^{0},u_{j}^{0})=\lambda_{i}^{0}b(u_{i}^{0},u_{j}^{0})=\lambda_{i}^{0}\delta_{ij},\ \ \ \ i,j=1, 2,\cdots,M,$$
\ \ \ \ \ \ \ \ \ \ \ \  Set $U_{J}^{0}={\rm span}\{u_{i}^{0}\}_{i=m}^{M},\Lambda_{J}^{0}=\{\lambda_{i}^{0}\}_{i=m}^{M}$ and $W_{J}^{0}={\rm span}\{u_{i}^{0}\}_{i=1}^{M}.$ \\
$\bm{Step\ 2}$. For $k=0,1,...$, solve parallel block preconditioned Jacobi-Davidson correction\\ 
\ \ \ \ \ \ \ \ \ \ \ \ equations:
\begin{equation}\label{Equation_CorrectionVariable}
t_{i}^{k}=(I-Q_{J}^{k})(B^{k}_{i})^{-1}r^{k}_{i},\ \ \ \ i=m,m+1,\cdots,M,
\end{equation}
\ \ \ \ \ \ \ \ \ \ \ \ \ where $r_{i}^{k}=\lambda_{i}^{k}u_{i}^{k}-A^{h}u_{i}^{k}$, the preconditioner $(B^{k}_{i})^{-1}$ is defined as \eqref{Equation_Preconditioner} and $Q_{J}^{k}$\\
\ \ \ \ \ \ \ \ \ \ \ \  is a $b(\cdot,\cdot)$-orthogonal projector onto $U_{J}^{k}$.\\
$\bm{Step\ 3}$. Solve eigenvalue problems in $W_{J}^{k+1}$ and take eigenpairs from the $m$-th to $M$-th:
\begin{equation}\label{Equation_SmallEigenvalueProblem}
a(u_{j}^{k+1},v)=\lambda_{j}^{k+1}b(u_{j}^{k+1},v)\ \ \ \ \forall\ v\in W_{J}^{k+1},\ \ \ b(u_{i}^{k+1},u_{j}^{k+1})=\delta_{ij} \ \ \ 1\leq i, j\leq n_{k+1},
\end{equation}
\ \ \ \ \ \ \ \ \ \ \ \ where $W_{J}^{k+1}=W_{J}^{k}+{\rm span}\{t_{i}^{k}\}_{i=m}^{M}$ and $n_{k+1}=\dim(W^{k+1}_{J})$.\\
\ \ \ \ \ \ \ \ \ \ \ \ Set $U_{J}^{k+1}=$span$\{u_{j}^{k+1}\}_{j=m}^{M}$ and $\Lambda_{J}^{k+1}=\{\lambda_{j}^{k+1}\}_{j=m}^{M}$.\\
$\bm{Step\ 4}$. If $\sum_{i=m}^{M}|\lambda_{i}^{k+1}-\lambda_{i}^{k}|<tol,$ then return $(\Lambda_{J}^{k+1},U_{J}^{k+1})$, otherwise goto $\bm{Step\ 2}$.\\
\hline
\hline
\end{tabular}
\end{table}

\begin{remark}
The assumption of $\lambda_{i}^{0}<\lambda_{i}^{H}$, $a(u_{i}^{0},u_{j}^{0})=\lambda_{i}^{0}b(u_{i}^{0},u_{j}^{0})=\lambda_{i}^{0}\delta_{ij}$, $i,j=1, 2,\cdots,M$ in Step 1 is realized through solving eigenvalue problems on $\mathcal{T}_{\widetilde{H}}$, obtained by refining  $\mathcal{T}_{H}$
\begin{equation}\label{Equation_InitialEigenvalueProblem}
A^{\widetilde{H}}u_{i}^{\widetilde{H}}=\lambda_{i}^{\widetilde{H}}u_{i}^{\widetilde{H}},\ \ \ \ b(u_{i}^{\widetilde{H}},u_{j}^{\widetilde{H}})=\delta_{ij},\ i,j=1,2,\cdots,M.
\end{equation}
and letting $u_{i}^{0}=u_{i}^{\widetilde{H}}$, $\lambda_{i}^{0}=\lambda_{i}^{\widetilde{H}}<\lambda_{i}^{H}$. In pratical computations, we may set $\widetilde{H}=\frac{H}{2}$.
\end{remark}

\begin{remark}
The two-level additive Schwarz method in this paper may be regarded as an essential extension of  the two-level BPJD method \cite{MR4735245}. For $1=m<M$ case, it reduces the method proposed in this paper into the two-level BPJD method proposed in \cite{MR4735245}.
\end{remark}

\begin{remark}
The two-level additive Schwarz method in this paper has greater advantages than that in \cite{MR4735245} for solving many eigenpairs of large-scale eigenvalue problems. Firstly, the proposed method is parallel in two ways: one is to compute the two-level additive Schwarz preconditioner to precondition the block Jacobi-Davidson correction equation, the other is to solve directly interior multiple and clustered eigenvalues so that this method may solve different clusters of eigenvalues simultaneously. Secondly, the dimension of the trial subspace $W_{J}^{k+1}$ is significantly less than that in \cite{MR4735245}, which may reduce the computer memory and accelerate to the solution of eigenvalue problems in $W_{J}^{k+1}$.
\end{remark}
\par For convenience of theoretical analysis, we first give a result for the Rayleigh-Ritz method. Define $\theta_{a}^{k}:=\theta_{a}$ $(U_{J}^{k},U_{J}^{h})$ and $\theta_{a,*}^{k}:=\theta_{a}(U^{k},U_{*}^{h}),$
where $U^{k}={\rm span}\{u_{i}^{k}\}_{i=1}^{M}$ and
$U_{*}^{h}=U_{L}^{h}\oplus U_{J}^{h}$. We also define $g_{i}^{k}:=\mu_{i}^{k}u_{i}^{k}-T^{h}u_{i}^{k}$ with $\mu_{i}^{k}=(\lambda_{i}^{k})^{-1}$. We call $g_{i}^{k}$ the residual of $T^{h}$.

\begin{lemma}\label{Lemma_estimate_UkUJ}
If Assumption \ref{Assumption1} holds, then $\theta_{a}^{k}\leq CH,$ where $C$ depends on $d_{m}^{-}$ and $d_{M}^{+}$.
\end{lemma}
\begin{proof}
Noting that $\lambda_{m}^{k}>\lambda_{m-1}^{k}$ in Algorithm 3.1 and $\lambda_{i}^{k}$ decreases monotonically and satisfies $\lambda_{i}^{k}-\lambda_{i}^{h}\leq \lambda_{i}^{0}-\lambda_{i}^{h}\leq \lambda_{i}^{H}-\lambda_{i}^{h}\leq CH^{2}$, $i=1,2,...,M$, we know that $\mu_{m-1}^{k}-\mu_{m}^{k}\geq C$ for sufficiently small $H$. Moreover, using Theorem 7 in \cite{MR2214751}, Assumption \ref{Assumption1} and Theorem \ref{theorem_priori}, we get
\begin{align}
\begin{aligned}\label{align_estimate_theta_a}
(\theta^{k}_{a})^{2}
&\leq \left(\frac{1}{\mu^{k}_{m-1}-\mu^{k}_{m}}+\frac{1}{\mu^{h}_{M}-\mu^{h}_{M+1}}\right)\sum_{i=m}^{M}(\mu_{i}^{h}-\mu_{i}^{k})
\leq C\sum_{i=m}^{M}(\mu_{i}^{h}-\mu_{i}^{k})\\
&\leq C\sum_{i=m}^{M}(\lambda_{i}^{k}-\lambda_{i}^{h})
\leq C\sum_{i=m}^{M}(\lambda_{i}^{0}-\lambda_{i}^{h})
\leq C\sum_{i=m}^{M}(\lambda_{i}^{H}-\lambda_{i}^{h})\leq CH^{2},
\end{aligned}
\end{align}
where $C$ depends on $d_{m}^{-}$ and $d_{M}^{+}$.
\end{proof}
\par According to Lemma 4.3 in \cite{MR4735245}, we have
\begin{align}\label{align_estimate_theta_a_1M}
\theta_{a,*}^{k}\leq CH,
\end{align}
We use the $a$-norm of the operator $(|||\cdot|||_{a})$ to  characterize $\theta_{a}^{k}$. We may prove
\begin{align}\label{align_theta_ak_def}
\begin{aligned} 
\theta_{a}^{k}
=|||(I-P_{J}^{k})Q_{J}^{h}|||_{a}
=|||Q_{J}^{h}(I-P_{J}^{k})|||_{a}
=|||Q_{J,\perp}^{h}P_{J}^{k}|||_{a}
=||| P_{J}^{k}Q_{J,\perp}^{h}|||_{a},
\end{aligned}
\end{align}
for sufficiently small $H$, where $P_{J}^{k}$ is the $a(\cdot,\cdot)$-orthogonal projector from $V^{h}$ onto $U_{J}^{k}$ and $Q^{h}_{J,\perp}$ is the $b(\cdot,\cdot)$-orthogonal projector from $V^{h}$ onto $U_{L}^{h}\oplus U_{R}^{h}$.
Similarly, 
\begin{align}\label{align_theta_ak_def_1M}
\begin{aligned} 
\theta_{a,*}^{k}
=|||(I-P^{k})Q_{*}^{h}|||_{a}
=|||Q_{*}^{h}(I-P^{k})|||_{a}
=|||Q_{R}^{h}P^{k}|||_{a}
=||| P^{k}Q_{R}^{h}|||_{a},
\end{aligned}
\end{align}
where $P^{k}$ is the $a(\cdot,\cdot)$-orthogonal projector from $V^{h}$ onto $U^{k}$ and $Q_{*}^{h}$ is the $b(\cdot,\cdot)$-orthogonal projector from $V^{h}$ onto $U_{*}^{h}$.

\par Next, we give two useful lemmas in the following analysis. 
\begin{lemma}\label{Lemma_QLQJ}
If Assumption \ref{Assumption1} holds, then for any $v^{H}\in V^{H}$, it holds that
\begin{align*}
||Q_{L}^{h}Q_{J}^{H}v^{H}||_{b}\leq CH^{2}||Q_{J}^{H}v^{H}||_{b},\ \ \ \ \ ||Q_{R}^{h}Q_{J}^{H}v^{H}||_{b}\leq CH^{2}||Q_{J}^{H}v^{H}||_{b}
\end{align*}
and
\begin{align*}
||Q_{L}^{h}Q_{J}^{H}v^{H}||_{a}\leq CH||Q_{J}^{H}v^{H}||_{a},\ \ \ \ \ ||Q_{R}^{h}Q_{J}^{H}v^{H}||_{a}\leq CH||Q_{J}^{H}v^{H}||_{a}.
\end{align*}
\end{lemma}
\begin{proof}
Combining Theorem \ref{theorem_priori} and the proof argument as in Lemma 4.1 in \cite{MR4735245}, we may obtain the above inequalities.
\end{proof}

\begin{lemma}\label{Lemma_lambdaik_Rq}
If Assumption 1 holds, then for sufficiently small $H$, we have
\begin{align*}
\lambda_{i}^{k}-Rq(Q_{J}^{h}u_{i}^{k})\leq CH^{2},\ \ \ \ \ \ \ \ 
\lambda_{i}^{k}-Rq(Q_{J}^{h}u_{i}^{k})\leq C||g_{i}^{k}||_{a}^{2},
\end{align*}
where $m\leq i\leq M$.
\end{lemma}
\begin{proof}
By the fact $a(g_{i}^{k},u_{i}^{k})=0$, we get
\begin{align*}
\left(\mu_{i}^{k}-Rt(Q_{J}^{h}u_{i}^{k})\right)||Q_{J}^{h}u_{i}^{k}||_{a}^{2}
=a((T^{h}-\mu_{i}^{k})Q_{J,\perp}^{h}u_{i}^{k},Q_{J,\perp}^{h}u_{i}^{k}).
\end{align*}
Moreover
\begin{align}\label{nu_de}\begin{aligned}
\lambda_{i}^{k}-Rq(Q_{J}^{h}u_{i}^{k})
=\frac{Rt(Q_{J}^{h}u_{i}^{k})-\mu_{i}^{k}}{Rt(Q_{J}^{h}u_{i}^{k})\mu_{i}^{k}}
=\frac{a((\mu_{i}^{k}-T^{h})Q_{J,\perp}^{h}u_{i}^{k},Q_{J,\perp}^{h}u_{i}^{k})}{||Q_{J}^{h}u_{i}^{k}||_{a}^{2}Rt(Q_{J}^{h}u_{i}^{k})\mu_{i}^{k}}.
\end{aligned}
\end{align}
For the numerator in \eqref{nu_de}, using \eqref{align_estimate_theta_a_1M}, \eqref{align_theta_ak_def_1M} and Assumption \ref{Assumption1}, then for sufficiently small $H$
\begin{align}\label{align_muik_Th_QRh}\begin{aligned}
a((\mu_{i}^{k}-T^{h})Q_{J,\perp}^{h}u_{i}^{k},Q_{J,\perp}^{h}u_{i}^{k})
&\leq a((\mu_{i}^{k}-T^{h})Q_{R}^{h}u_{i}^{k},Q_{R}^{h}u_{i}^{k})\\
&\leq \mu_{i}^{k}||Q_{R}^{h}u_{i}^{k}||_{a}^{2}
=\mu_{i}^{k}||Q_{R}^{h}P^{k}u_{i}^{k}||_{a}^{2}
\leq C\mu_{i}^{k}H^{2}.
\end{aligned}
\end{align}
For the denominator in \eqref{nu_de}, by Lemma \ref{Lemma_estimate_UkUJ} and \eqref{align_theta_ak_def}, we get
\begin{align*}
||Q_{J}^{h}u_{i}^{k}||_{a}^{2}
&=||u_{i}^{k}||_{a}^{2}-||Q_{J,\perp}^{h}u_{i}^{k}||_{a}^{2}
=\lambda_{i}^{k}-||Q_{J,\perp}^{h}P_{J}^{k}u_{i}^{k}||_{a}^{2}\\
&\geq \lambda_{i}^{k}(1-(\theta_{a}^{k})^{2})
\geq \lambda_{i}^{k}(1-CH^{2}),
\end{align*}
which, together with \eqref{nu_de} and \eqref{align_muik_Th_QRh}, yields the first inequality.
Next, we consider the following operator eigenvalue problem
\begin{align*}
Q_{R}^{h}(\mu_{i}^{k}-T^{h})Q_{R}^{h}v=\beta Q_{R}^{h}(\mu_{i}^{k}-T^{h})^{2}Q_{R}^{h}v.
\end{align*}
By the Rayleigh quotient technique, we get
\begin{align}\label{align_operator_eig}
\beta_{max}
:=\max_{w\in U_{R}^{h}, w\ne 0 }\frac{a((\mu_{i}^{k}-T^{h})Q_{R}^{h}w,Q_{R}^{h}w)}{a((\mu_{i}^{k}-T^{h})^{2}Q_{R}^{h}w,Q_{R}^{h}w)}
=\frac{1}{\mu_{i}^{k}-\mu_{M+1}^{h}}.
\end{align}
Therefore, by \eqref{align_muik_Th_QRh} and \eqref{align_operator_eig}, we have
\begin{align*}
a((\mu_{i}^{k}-T^{h})Q_{J,\perp}^{h}u_{i}^{k},Q_{J,\perp}^{h}u_{i}^{k})
&\leq a((\mu_{i}^{k}-T^{h})Q_{R}^{h}u_{i}^{k},Q_{R}^{h}u_{i}^{k})\\
&\leq \beta_{max}a((\mu_{i}^{k}-T^{h})^{2}Q_{R}^{h}u_{i}^{k},Q_{R}^{h}u_{i}^{k})
\leq  \beta_{max}||g_{i}^{k}||_{a}^{2},
\end{align*}
which, together with estimates of the numerator, completes the proof.
\end{proof}

\section{Convergence analysis}\label{sec4}
\par In this section, we give a rigorous convergence analysis for Algorithm 3.1. Using some mathematical tools, including the stable decomposition technique in high frequency subspace, the maximum angle estimate argument in low frequency subspace and constructions of auxiliary eigenvalue problems, we may prove that the convergence factor of the proposed method is independent of the fine mesh size $h$ and the gaps among the interior eigenvalues $\{\lambda_{i}\}_{i=m}^{M}$, but depends on the ratio $\frac{\delta}{H}$ and the gaps $d_{m}^{-}$ and $d_{M}^{+}$. First, we present the main result of this paper in the following.

\begin{theorem}\label{Theorem_PJD_conver_rate}
If Assumption \ref{Assumption1} and Assumption \ref{Assumption2} hold, then
\begin{equation}\label{Equation_Theorem eigenvalues}
\sum_{i=m}^{M}(\lambda_{i}^{k+1}-\lambda_{i}^{h})\leq \gamma \sum_{i=m}^{M}(\lambda_{i}^{k}-\lambda_{i}^{h}),
\end{equation}
and
\begin{align}\label{Equation_Theorem eigenvector_a}
(\theta_{a}^{k})^{2}\leq C\gamma^{k},
\end{align}
where the convergence factor $\gamma$ is independent of the fine mesh size $h$ and the gaps among the eigenvalues $\{\lambda_{i}\}_{i=m}^{M}$. The factor $\gamma=c(H)\rho(\frac{\delta}{H},d_{m}^{-},d_{M}^{+})$, with $c(H)\to 1$, as $H\to 0$
and $0<\rho(\frac{\delta}{H},d_{m}^{-},d_{M}^{+})$ $<1$.
\end{theorem}

\begin{remark}
For convenience of the readers, we only give a theoretical analysis for the proposed method for solving the Laplacian eigenvalue problem. It is easy to extend the results to biharmonic eigenvalue problems.   
\end{remark}

\par Before proving Theorem \ref{Theorem_PJD_conver_rate}, we first give the idea of the proof. For notational convenience, we denote by $Q^{k}:V^{h}\to U^{k}$ a $b(\cdot,\cdot)$-orthogonal projector.
Since $U_{J}^{k}\subset U^{k}\subset W_{J}^{k}$, we know that for $i=m,m+1,\cdots,M$,
\begin{align}\label{Align_widetildeuk1}
\widetilde{u}_{i}^{k+1}
=u_{i}^{k}+\alpha_{i}^{k}(I-Q^{k})t_{i}^{k}
=(u_{i}^{k}-\alpha_{i}^{k}Q^{k}t_{i}^{k})+\alpha_{i}^{k}t_{i}^{k}
\in W_{J}^{k+1},
\end{align}
where $\{\alpha_{i}^{k}\}_{i=m}^{M}$ are some undetermined parameters.
Consider the first auxiliary eigenvalue problem as follows
\begin{align}\label{Align_JD_AuxiEigen}
a(\hat{u}_{i}^{k+1},v)=\hat{\lambda}_{i}^{k+1}b(\hat{u}_{i}^{k+1},v)\ \ \ \ \forall\ v\in \widetilde{U}^{k+1}_{J}\oplus {\rm span}\{u_{i}^{k}\}_{i=1}^{m-1},
\end{align}
where $\widetilde{U}^{k+1}_{J}={\rm span}\{\widetilde{u}_{i}^{k+1}\}_{i=m}^{M}.$
By the Courant-Fisher principle, we have
\begin{align}\label{hat_eig_inequ}
\lambda_{i}^{k+1}\leq \hat{\lambda}_{i}^{k+1},\ \ \ \ i=m,m+1,...,M.
\end{align}
Therefore, it suffices to prove that
$\sum_{i=m}^{M}(\hat{\lambda}_{i}^{k+1}-\lambda_{i}^{h})\leq \gamma \sum_{i=m}^{M}(\lambda_{i}^{k}-\lambda_{i}^{h}),$ which may complete the proof of Theorem \ref{Theorem_PJD_conver_rate}.

\subsection{Error reduction in high frequency subspace}
\par In this section, we analyze error reduction in high frequency subspace $U_{R}^{h}$. Define $(w,v)_{E_{i,R}^{k}}:=b((A^{h}-\lambda_{i}^{k})w,v)$ $\forall\ w,v\in V^{h}$ which constructs an inner product $U_{R}^{h}$. We denote by $||v||_{E_{i,R}^{k}}^{2}:=(v,v)_{E_{i,R}^{k}},\ \forall v\in V^{h}$ which satisfies $||v||_{E_{i,R}^{k}}\leq ||v||_{a}\leq C||v||_{E_{i,R}^{k}}$ for all $v\in U_{R}^{h}$.
For notational convenience, we denote by $\widetilde{e}_{i,R}^{k+1}=-Q_{R}^{h}\widetilde{u}_{i}^{k+1}$, $e_{i,L}^{k}:=-Q_{L}^{h}u_{i}^{k},\  u_{i,J}^{k}:=Q_{J}^{h}u_{i}^{k}$ and $ e_{i,R}^{k}:=-Q_{R}^{h}u_{i}^{k}.$

From \eqref{Equation_CorrectionVariable}, \eqref{Align_widetildeuk1} and the decomposition \eqref{Align_fine_decom}, we deduce
\begin{align}\label{Align_error_URh}
\begin{aligned}
\widetilde{e}_{i,R}^{k+1}
&=e_{i,R}^{k}-\alpha_{i}^{k}Q_{R}^{h}(I-Q^{k})(I-Q_{J}^{k})(B_{i}^{k})^{-1}r_{i}^{k}\\
&=e_{i,R}^{k}-\alpha_{i}^{k}Q_{R}^{h}(B_{i}^{k})^{-1}r_{i}^{k}
+\alpha_{i}^{k}Q_{R}^{h}Q_{J}^{k}(B_{i}^{k})^{-1}r_{i}^{k}\\
&\ \ \ \ +\alpha_{i}^{k}Q_{R}^{h}Q^{k}(I-Q_{J}^{k})(B_{i}^{k})^{-1}r_{i}^{k}\\
&=\left(e_{i,R}^{k}-\alpha_{i}^{k}Q_{R}^{h}(B_{i}^{k})^{-1}(A^{h}-\lambda_{i}^{k})e_{i,R}^{k}\right)\\
&\ \ \ \ +\alpha_{i}^{k}Q_{R}^{h}(B_{i}^{k})^{-1}(A^{h}-\lambda_{i}^{k})(u_{i,J}^{k}-e_{i,L}^{k})\\
&\ \ \ \  +\{\alpha_{i}^{k}Q_{R}^{h}Q_{J}^{k}(B_{i}^{k})^{-1}r_{i}^{k}
+\alpha_{i}^{k}Q_{R}^{h}Q^{k}(I-Q_{J}^{k})(B_{i}^{k})^{-1}r_{i}^{k}\}\\
&=:I_{i,1}^{k}+I_{i,2}^{k}+I_{i,3}^{k}.
\end{aligned}
\end{align}
\par For three items in \eqref{Align_error_URh}, the estimation methods are the same as the estimation method of (5.8) in \cite{MR4735245}, and the results are given directly here.

\begin{lemma}\label{Theorem_Gik_norm_estimate}
Under the same assumptions as in Theorem \ref{Theorem_PJD_conver_rate}, for sufficiently small $\alpha_{i}^{k}$, the operator $G_{i,R}^{k}$ satisfies
\begin{align*}
||G_{i,R}^{k}v||_{E_{i,R}^{k}}\leq \left(1-C_{i}\frac{\delta}{H}\right)||v||_{E_{i,R}^{k}}
\ \ \ \ \forall\ v\in U_{R}^{h},\ i=m,m+1,\cdots,M,
\end{align*}
where $G_{i,R}^{k}=I-\alpha_{i}^{k}Q_{R}^{h}(B_{i}^{k})^{-1}(A^{h}-\lambda_{i}^{k})$, $I_{i,1}^{k}=G_{i,R}^{k}e_{i,R}^{k}$ and the positive constants $C_{i}<1$ depend on the gap $d_{M}^{+}$, but not $d_{m}^{-}$.
\end{lemma}

\begin{lemma}\label{Lemma_Estimate_I2ik_I3ik_I4ik}
Under the same assumptions as in Theorem \ref{Theorem_PJD_conver_rate}, we have
\begin{align*}
||I_{i,2}^{k}||_{E_{i,R}^{k}}&\leq CH^{2}||g_{i}^{k}||_{a},
\end{align*}
and
\begin{align*}
||I_{i,3}^{k}||_{E_{i,R}^{k}}&\leq CH||e_{i,R}^{k}||_{E_{i,R}^{k}}+CH^{2}||g_{i}^{k}||_{a}.
\end{align*}
\end{lemma}

\begin{remark}
We estimate the third term $I_{i,3}^{k}$ in \eqref{Align_error_URh}. By Lemma \ref{Lemma_estimate_UkUJ}, \eqref{align_estimate_theta_a_1M}, \eqref{align_theta_ak_def}, \eqref{align_theta_ak_def_1M}
and the fact that $Q_{R}^{h}Q_{J,\perp}^{h} P_{J}^{k}Q_{J}^{k}=Q_{R}^{h}Q_{J}^{k}$, we have
\begin{align*}
||Q_{R}^{h}Q_{J,\perp}^{h} P_{J}^{k}Q_{J}^{k}(B_{i}^{k})^{-1}r_{i}^{k}||_{a}
&\leq C||Q_{J,\perp}^{h}P_{J}^{k}Q_{J}^{k}(B_{i}^{k})^{-1}r_{i}^{k}||_{a}\\
&\leq C\theta_{a}^{k}||Q_{J}^{k}(B_{i}^{k})^{-1}r_{i}^{k}||_{a}
\leq CH||(B_{i}^{k})^{-1}r_{i}^{k}||_{a},
\end{align*}
and
\begin{align*}
||Q_{R}^{h}Q^{k}(I-Q_{J}^{k})(B_{i}^{k})^{-1}r_{i}^{k}||_{a}
&= ||Q_{R}^{h}P^{k}Q^{k}(I-Q_{J}^{k})(B_{i}^{k})^{-1}r_{i}^{k}||_{a}\\
&\leq C\theta_{a,*}^{k}||Q^{k}(I-Q_{J}^{k})(B_{i}^{k})^{-1}r_{i}^{k}||_{a}
\leq CH||(B_{i}^{k})^{-1}r_{i}^{k}||_{a},
\end{align*}
which, together with the techniques for estimating $||(B_{i}^{k})^{-1}r_{i}^{k}||_{a}$ in \cite{MR4735245}, yields the results of Lemma \ref{Lemma_Estimate_I2ik_I3ik_I4ik}.
\end{remark}

\par Now we give a main result in high frequency subspace $U_{R}^{h}$. 
\begin{theorem}\label{Theorem_eirk1_eirk_gik}
Under the same assumptions as in Theorem \ref{Theorem_PJD_conver_rate}, for sufficiently small $\alpha_{i}^{k}$, then
\begin{align*}
||\widetilde{e}_{i,R}^{k+1}||_{E_{i,R}^{k}}
\leq \gamma_{i}||e_{i,R}^{k}||_{E_{i,R}^{k}}+
CH^{2}||g_{i}^{k}||_{a},\ \ \ \ i=m,m+1,...,M,
\end{align*}
where $\gamma_{i}=1-C_{i}\frac{\delta}{H}+CH$, $C_{i}$ is defined in Lemma \ref{Theorem_Gik_norm_estimate}.
\end{theorem}
\begin{proof}
Combining \eqref{Align_error_URh}, Lemma \ref{Theorem_Gik_norm_estimate} and Lemma \ref{Lemma_Estimate_I2ik_I3ik_I4ik}, we get the result of this theorem.
\end{proof}

\subsection{Error estimate in low frequency subspace}
\par In this subsection, we estimate the error in low frequency subspace $U_{L}^{h}$. For notational convenience, we define
$(v,w)_{E_{i,L}^{k}}:=\lambda_{i}^{k}b(v,w)-a(v,w),\ i=m,m+1,\cdots,M$ for all $v,w\in V^{h}$. If Assumption \ref{Assumption1} holds, it is easy to see that $(\cdot,\cdot)_{E_{i,L}^{k}}$ constructs an inner product over $U_{L}^{h}$ for sufficiently small $H$. We denote by $||\cdot||_{E_{i,L}^{k}}^{2}:=(\cdot,\cdot)_{E_{i,L}^{k}}$. Now we give an estimate for $||e_{i,L}^{k}||_{E_{i,L}^{k}}$ in the following.
\begin{theorem}\label{Lemma_eiLk_estimate}
If Assumption \ref{Assumption1} holds, for sufficiently small $H$, we have
\begin{align*}
||e_{i,L}^{k}||_{E_{i,L}^{k}}\leq CH||g_{i}^{k}||_{a},\ \ \ \ i=m,m+1,...,M.
\end{align*}
\end{theorem}
\begin{proof}
By the facts that $(I-P_{J}^{k})g_{i}^{k}=g_{i}^{k}$ and $Q_{*}^{h}Q_{L}^{h}=Q_{L}^{h}$, we deduce
\begin{align}\label{Align_b_L1}
\begin{aligned}
||e_{i,L}^{k}||_{E_{i,L}^{k}}^{2}
&=b((\lambda_{i}^{k}-A^{h})Q_{L}u_{i}^{k},e_{i,L}^{k})
=\lambda_{i}^{k}a((T^{h}-\mu_{i}^{k})u_{i}^{k},e_{i,L}^{k})\\
&=-\lambda_{i}^{k}a((I-P_{k})g_{i}^{k},Q_{*}^{h}e_{i,L}^{k})\\
&\leq C||g_{i}^{k}||_{a}||(I-P_{k})Q_{*}^{h}e_{i,L}^{k}||_{a}
\leq C\theta_{a,*}^{k}||g_{i}^{k}||_{a}||e_{i,L}^{k}||_{a}.
\end{aligned}
\end{align}
Noting that
$b((\lambda_{i}^{k}-A^{h})e_{i,L}^{k},e_{i,L}^{k})=\left(\lambda_{i}^{k}-Rq(e_{i,L}^{k})\right)||e_{i,L}^{k}||_{b}^{2},$ we get
\begin{align*}
||e_{i,L}^{k}||_{a}^{2}
\leq C||e_{i,L}^{k}||_{b}^{2}
&\leq \frac{C\theta_{a,*}^{k}||g_{i}^{k}||_{a}||e_{i,L}^{k}||_{a}}{\lambda_{i}^{k}-Rq(e_{i,L}^{k})}
\leq \frac{C\theta_{a,*}^{k}||g_{i}^{k}||_{a}||e_{i,L}^{k}||_{a}}{\lambda_{i}^{k}-\lambda_{m-1}^{h}}.
\end{align*}
That is $||e_{i,L}^{k}||_{a}\leq C\theta_{a,*}^{k}||g_{i}^{k}||_{a}$. Therefore, following \eqref{Align_b_L1}, we have
\begin{align*}
||e_{i,L}^{k}||_{E_{i,L}^{k}}^{2}
\leq \lambda_{i}^{k}\theta_{a,*}^{k}||g_{i}^{k}||_{a}||e_{i,L}^{k}||_{a}
\leq C(\theta_{a,*}^{k})^{2}||g_{i}^{k}||_{a}^{2},
\end{align*}
which, together with \eqref{align_estimate_theta_a_1M}, completes the proof.
\end{proof}

\subsection{Error analysis for interior eigenvalues}
\par In this subsection, we focus on analyzing the total error (summation form of the error of the targeted eigenvalues) for discrete interior eigenvalues $\{\lambda_{i}^{h}\}_{i=m}^{M}$. The results of the high frequency analysis (Theorem \ref{Theorem_eirk1_eirk_gik}) and the low frequency analysis (Theorem \ref{Lemma_eiLk_estimate}) are very important for the analysis in this subsection.
\par Based on Theorem \ref{Theorem_eirk1_eirk_gik}, for any $i=m,m+1,...,M,$ we have
\begin{align}\label{Align_eik1_sqrt1}
\begin{aligned}
||\widetilde{e}_{i,R}^{k+1}||^{2}_{E_{i,R}^{k}}
\leq (\gamma_{i}^{2}+CH^{2})||e_{i,R}^{k}||_{E_{i,R}^{k}}^{2}+CH^{2}||g_{i}^{k}||_{a}^{2}.
\end{aligned}
\end{align}
Since $b(r_{i}^{k},u_{i}^{k})=0$ and $u_{i}^{k}=u_{i,J}^{k}-e_{i,L}^{k}-e_{i,R}^{k}, i=m,m+1,\cdots,M$, we get
\begin{align}\label{Align_eik1_sqrt2}
||e_{i,R}^{k}||_{E_{i,R}^{k}}^{2}
=||e_{i,L}^{k}||_{E_{i,L}^{k}}^{2}+(\lambda_{i}^{k}-Rq(u_{i,J}^{k}))||u_{i,J}^{k}||_{b}^{2},
\end{align}
which, together with \eqref{Align_eik1_sqrt1} and Thorem \ref{Lemma_eiLk_estimate}, yields
\begin{align}\label{upper_bound_w_eiRk1}
||\widetilde{e}_{i,R}^{k+1}||^{2}_{E_{i,R}^{k}}
&\leq (\gamma_{i}^{2}+CH^{2})(\lambda_{i}^{k}-Rq(u_{i,J}^{k}))||u_{i,J}^{k}||_{b}^{2}+CH^{2}||g_{i}^{k}||_{a}^{2}.
\end{align}
In addition, using \eqref{Align_fine_decom}, we may get the lower bound for $||\widetilde{e}_{i,R}^{k+1}||^{2}_{E_{i,R}^{k}}$ as
\begin{align}\label{Align_eik1_estimate}
\begin{aligned}
||\widetilde{e}_{i,R}^{k+1}||^{2}_{E_{i,R}^{k}}
&=(Rq(\widetilde{u}_{i}^{k+1})-\lambda_{i}^{k})||\widetilde{u}_{i}^{k+1}||_{b}^{2}
+(\lambda_{i}^{k}-Rq(\widetilde{u}_{i,J}^{k+1}))||\widetilde{u}_{i,J}^{k+1}||_{b}^{2}\\
&\ \ \ \ +(\lambda_{i}^{k}-Rq(\widetilde{e}_{i,L}^{k+1}))||\widetilde{e}_{i,L}^{k+1}||_{b}^{2}\\
&\geq (Rq(\widetilde{u}_{i}^{k+1})-Rq(u_{i,J}^{k}))||\widetilde{u}_{i}^{k+1}||_{b}^{2}
+(Rq(u_{i,J}^{k})-\lambda_{i}^{k})||\widetilde{e}_{i,R}^{k+1}||_{b}^{2}\\
&\ \ \ \ +(Rq(u_{i,J}^{k})-Rq(\widetilde{u}_{i,J}^{k+1}))||\widetilde{u}_{i,J}^{k+1}||_{b}^{2}.
\end{aligned}
\end{align}
\par Based on the upper and lower bound for $||\widetilde{e}_{i,R}^{k+1}||^{2}_{E_{i,R}^{k}}$, we have the following result. 
\begin{lemma}\label{Lemma_gamma_max}
Under the same Assumptions as Theorem \ref{Theorem_PJD_conver_rate}, for sufficiently small $H$, we have
\begin{align*}
\sum_{i=m}^{M}(Rq(\widetilde{u}_{i}^{k+1})-Rq(u_{i,J}^{k}))
\leq \gamma_{\max}\sum_{i=m}^{M}\eta_{i}^{k}+CH\sum_{i=m}^{M}||g_{i}^{k}||_{a}^{2},
\end{align*}
where $\gamma_{\max}=\max_{m\leq i\leq M} \gamma_{i}^{2}+CH$ and if $\lambda_{i}^{k}-Rq(u_{i,J}^{k})>0$ then $\eta_{i}^{k}=\lambda_{i}^{k}-Rq(u_{i,J}^{k})$, else  $\eta_{i}^{k}=\lambda_{i}^{k}-\lambda_{i}^{h}$.
\end{lemma}

\begin{proof}
For $m\leq i\leq M$, from \eqref{Align_eik1_estimate}, we obtain
\begin{align*}
(Rq(\widetilde{u}_{i}^{k+1})-Rq(u_{i,J}^{k}))||\widetilde{u}_{i}^{k+1}||_{b}^{2}
&\leq \left(||\widetilde{e}_{i,R}^{k+1}||^{2}_{E_{i,R}^{k}}+(\lambda_{i}^{k}-Rq(u_{i,J}^{k}))||\widetilde{e}_{i,R}^{k+1}||_{b}^{2}\right)\\
&\ \ \ \ +(Rq(\widetilde{u}_{i,J}^{k+1})-Rq(u_{i,J}^{k}))||\widetilde{u}_{i,J}^{k+1}||_{b}^{2}
=:J_{1}+J_{2}.
\end{align*}
For $\lambda_{i}^{k}-Rq(u_{i,J}^{k})>0$, by Lemma \ref{Lemma_lambdaik_Rq}, Thorem \ref{Lemma_eiLk_estimate}, \eqref{Align_eik1_sqrt1} and \eqref{Align_eik1_sqrt2}, we have
\begin{align*}
J_{1}
&\leq (1+CH^{2})||\widetilde{e}_{i,R}^{k+1}||^{2}_{E_{i,R}^{k}}
\leq c(H)(\gamma_{i}^{2}+CH^{2})||e_{i,R}^{k}||_{E_{i,R}^{k}}^{2}+CH^{2}||g_{i}^{k}||_{a}^{2}\\
&\leq c(H)(\gamma_{i}^{2}+CH^{2})(\lambda_{i}^{k}-Rq(u_{i,J}^{k}))+CH^{2}||g_{i}^{k}||_{a}^{2}.
\end{align*}
For $\lambda_{i}^{k}-Rq(u_{i,J}^{k})\leq 0$, by Thorem \ref{Lemma_eiLk_estimate}, \eqref{Align_eik1_sqrt1} and \eqref{Align_eik1_sqrt2}, we get
\begin{align*}
J_{1}&\leq ||\widetilde{e}_{i,R}^{k+1}||^{2}_{E_{i,R}^{k}}
\leq (\gamma_{i}^{2}+CH^{2})\{||e_{i,L}^{k}||_{E_{i,L}^{k}}^{2}+(\lambda_{i}^{k}-Rq(u_{i,J}^{k}))||u_{i,J}^{k}||_{b}^{2}\}
\\
&\ \ \ \ +CH^{2}||g_{i}^{k}||_{a}^{2}
\leq (\gamma_{i}^{2}+CH^{2})(\lambda_{i}^{k}-\lambda_{i}^{h})+CH^{2}||g_{i}^{k}||_{a}^{2}.
\end{align*}
\par For notational convenience, we denote by $w_{i}^{k}:=Q_{J}^{h}(I-Q^{k})(I-Q_{J}^{k})(B_{i}^{k})^{-1}r_{i}^{k}$.
Note that
$||\widetilde{u}_{i,J}^{k+1}||_{b}^{2}
=b(u_{i,J}^{k},u_{i,J}^{k})+2\alpha_{i}^{k}b(u_{i,J}^{k},w_{i}^{k})+(\alpha_{i}^{k})^{2}b(w_{i}^{k},w_{i}^{k})$
and
$||\widetilde{u}_{i,J}^{k+1}||_{a}^{2}
=a(u_{i,J}^{k},u_{i,J}^{k})+2\alpha_{i}^{k}a(u_{i,J}^{k},w_{i}^{k})+(\alpha_{i}^{k})^{2}a(w_{i}^{k},w_{i}^{k}).$
Using the similar argument as \eqref{align_estimate_theta_a_1M} and \eqref{align_theta_ak_def_1M}, we may have $\theta_{b}^{k}(U^{k},U_{*}^{h})=|||Q_{*}^{h}(I-Q^{k})|||_{b}\leq CH$ and let $\theta_{b,*}^{k}:=\theta_{b}^{k}(U^{k},U_{*}^{h})$.  Moreover,
\begin{align*}
J_{2}
&=||\widetilde{u}_{i,J}^{k+1}||_{a}^{2}-Rq(u_{i,J}^{k})||\widetilde{u}_{i,J}^{k+1}||_{b}^{2}\\
&=2\alpha_{i}^{k}b((A^{h}-\lambda_{i}^{k})u_{i,J}^{k},w_{i}^{k})+2\alpha_{i}^{k}b((\lambda_{i}^{k}-Rq(u_{i,J}^{k}))u_{i,J}^{k},w_{i}^{k})\\
&\ \ \ \ +(\alpha_{i}^{k})^{2}b((A^{h}-Rq(u_{i,J}^{k}))w_{i}^{k},w_{i}^{k})\\
&\leq C\theta_{b,*}^{k}||g_{i}^{k}||_{a}||(B_{i}^{k})^{-1}r_{i}^{k}||_{b}
+C\theta_{b,*}^{k}(\lambda_{i}^{k}-Rq(u_{i,J}^{k}))||(B_{i}^{k})^{-1}r_{i}^{k}||_{b}\\
&\ \ \ \ +C(\theta_{b,*}^{k})^{2}||(B_{i}^{k})^{-1}r_{i}^{k}||^{2}_{b},
\end{align*}
which, together with the fact $||(B_{i}^{k})^{-1}r_{i}^{k}||_{b}\leq C||e_{i,R}^{k}||_{E_{i,R}^{k}}+CH^{2}||g_{i}^{k}||_{a}$, yields
\begin{align*}
J_{2}\leq CH\eta_{i}^{k}+CH||g_{i}^{k}||_{a}^{2}.
\end{align*}
Further, by the fact that $||\widetilde{u}_{i}^{k+1}||_{b}\geq 1$, we obtain
\begin{align*}
(Rq(\widetilde{u}_{i}^{k+1})-Rq(u_{i,J}^{k}))
&\leq \frac{1}{||\widetilde{u}_{i}^{k+1}||_{b}^{2}}(J_{1}+J_{2})\leq J_{1}+J_{2}\\
&\leq c(H)(\gamma_{i}^{2}+CH^{2})||e_{i,R}^{k}||_{E_{i,R}^{k}}^{2}+CH\eta_{i}^{k}+CH||g_{i}^{k}||_{a}^{2}\\
&\leq c(H)(\gamma_{i}^{2}+CH^{2})\eta_{i}^{k}+CH||g_{i}^{k}||_{a}^{2},
\end{align*}
which completes the proof.
\end{proof}

\par Later on, taking advantage of Lemma 5 in \cite{Ovtchinnikov}, it is straightforward to establish an estimate for $\sum_{i=m}^{M}(\hat{\lambda}_{i}^{k+1}-Rq(\widetilde{u}_{i}^{k+1}))$.
We omit the proof for brevity.
\begin{lemma}\label{Lemma_estimate_lam_kh}
Under the same Assumption as Theorem \ref{Theorem_PJD_conver_rate}, it holds that
\begin{align*}
\sum_{i=m}^{M}(\hat{\lambda}_{i}^{k+1}-Rq(\widetilde{u}_{i}^{k+1}))
\leq CH\sum_{i=m}^{M}\eta_{i}^{k} +CH\sum_{i=m}^{M}||g_{i}^{k}||_{a}^{2}.
\end{align*}
\end{lemma}

\subsection{Upper bound for residuals of $T^{h}$}
\par In this subsection, we give an upper estimate with $L^{2}$-orthogonal projection for $\sum_{i=m}^{M}||g_{i}^{k}||^{2}_{a}$. 
For $m\leq i \leq M$, define
$\bar{u}_{i}^{k+1}:=u_{i}^{k}+\tau (I-Q^{k})t_{i}^{k}$,
where $\tau$ is a parameter which depends on the number $N_{0}$. Then we construct the auxiliary eigenvalue problem as follows
\begin{align}\label{align_second_auxi_eig}
a(\check{u}_{i}^{k+1},v)=\check{\lambda}_{i}^{k+1}b(\check{u}_{i}^{k+1},v)\ \ \ \ \forall\ v\in \bar{U}^{k+1},
\end{align}
where $\bar{U}^{k+1}={\rm span}\{\bar{u}_{i}^{k+1}\}_{i=m}^{M}\oplus {\rm span}\{u_{i}^{k}\}_{i=1}^{m-1}$. Since $\bar{U}^{k}\subset W_{J}^{k+1}$, we know 
\begin{align}\label{check_eig_inequ}
    \lambda_{i}^{k+1}\leq \check{\lambda}_{i}^{k+1}.
\end{align}
\par Through analyzing \eqref{align_second_auxi_eig}, we have the following. \begin{lemma}\label{Lemma_estimate_gik}
Under the same Assumptions as Theorem \ref{Theorem_PJD_conver_rate}, for sufficiently small $H$, we have
\begin{align*}
\sum_{i=m}^{M}||g_{i}^{k}||_{a}^{2}\leq C\sum_{i=m}^{M}(\lambda_{i}^{k}-\check{\lambda}_{i}^{k+1}),
\end{align*}
where $\check{\lambda}_{i}^{k+1}$ is defined in \eqref{align_second_auxi_eig}.
\end{lemma}
\begin{proof}
The matrix form of \eqref{align_second_auxi_eig} defined in $\bar{U}^{k+1}$ may be rewritten as 
\begin{align*}
\left(
\begin{matrix}
A_{11}&A_{12}\\
A_{21}&A_{22}
\end{matrix}
\right)y_{i}
=\check{\lambda}_{i}^{k+1}
\left(
\begin{matrix}
B_{11}&0\\
0&B_{22}
\end{matrix}
\right)y_{i},
\end{align*}
where $(A_{11},B_{11})$ and $(A_{22},B_{22})$ correspond to ${\rm span}\{\bar{u}_{i}^{k+1}\}_{i=m}^{M}$ and $ {\rm span}\{u_{i}^{k}\}_{i=1}^{m-1}$, respectively. For convenience, we denote by $\lambda_{i,1}\ (m\leq i\leq M)$ the eigenvalues of the matrix pair $(A_{11},B_{11})$. Note that $\sum_{i=m}^{M}(\lambda_{i}^{k}-\check{\lambda}_{i}^{k+1})
=\sum_{i=m}^{M}(\lambda_{i}^{k}-\lambda_{i,1})-\sum_{i=m}^{M}(\check{\lambda}_{i}^{k+1}-\lambda_{i,1}).$
Therefore, it suffices to prove that
\begin{align}\label{Lemma_suffice_gika}
\sum_{i=m}^{M}(\lambda_{i}^{k}-\lambda_{i,1})\geq C\sum_{i=m}^{M}||g_{i}^{k}||_{a}^{2}\ \ \text{and}
\ \ \sum_{i=m}^{M}(\check{\lambda}_{i}^{k+1}-\lambda_{i,1})
\leq C\epsilon(H)\sum_{i=m}^{M}||g_{i}^{k}||_{a}^{2},
\end{align}
where $\epsilon(H)\to 0$, as $H\to 0$.
\par We first prove the left inequality of \eqref{Lemma_suffice_gika}. For $m\leq i,j\leq M$, denoting by $q_{i}^{k}=(I-Q^{k})t_{i}^{k}$, we have
\begin{align*}
b(\bar{u}_{i}^{k+1},\bar{u}_{j}^{k+1})
=b(u_{i}^{k}+\tau q_{i}^{k},u_{j}^{k}+\tau q_{j}^{k})
=\delta_{ij}+\tau^{2}b(q_{i}^{k},q_{j}^{k}),
\end{align*}
which means $B_{11}\geq I_{11}$ with $I_{11}$ being an identity matrix of same size. Furthermore, noting that $a(u_{i}^{k},q_{j}^{k})
=b(A^{h}u_{i}^{k}-\lambda_{i}^{k}u_{i}^{k},q_{j}^{k})
=-b(r_{i}^{k},q_{j}^{k})=-b(r_{i}^{k},(B_{j}^{k})^{-1}r_{j}^{k})$, we get
\begin{align*}
a(\bar{u}_{i}^{k+1},\bar{u}_{j}^{k+1})
&=\lambda_{i}^{k}\delta_{ij}-\tau\left(b(r_{i}^{k},(B_{j}^{k})^{-1}r_{j}^{k})+b((B_{i}^{k})^{-1}r_{i}^{k},r_{j}^{k})\right)
+\tau^{2}a(q_{i}^{k},q_{j}^{k})\\
&=:a^{1}_{ij}-\tau a^{2}_{ij}+\tau^{2}a^{3}_{ij}.
\end{align*}
For convenience, we denote by 
$A_{1}=(a_{ij}^{1}),\ A_{2}=(a_{ij}^{2})$ and $A_{3}=(a_{ij}^{3}),$
so we have $A_{11}=A_{1}-\tau A_{2}+\tau^{2}A_{3}.$
Moreover, by the fact that $A_{11}$ is a symmetric semi-positive definite matrix, we deduce
\begin{align*}
\sum_{i=m}^{M}(\lambda_{i}^{k}-\lambda_{i,1})
&={\rm Tr}(A_{1})-{\rm Tr}(B_{11}^{-1}A_{11})
\geq {\rm Tr}(A_{1})-{\rm Tr}(A_{11})
=\tau {\rm Tr}(A_{2})-\tau^{2}{\rm Tr}(A_{3})\\
&=2\tau \sum_{i=m}^{M}b((B_{i}^{k})^{-1}r_{i}^{k},r_{i}^{k})-\tau^{2}\sum_{i=m}^{M}a(q_{i}^{k},q_{i}^{k}).
\end{align*}
On one hand, by Assumption \ref{Assumption2},
$||q_{i}^{k}||_{a}^{2}\leq C||(B_{i}^{k})^{-1}r_{i}^{k}||_{a}^{2}
\leq C_{N_{0}}b((B_{i}^{k})^{-1}r_{i}^{k},r_{i}^{k})$ holds.
On the other hand, we may give a lower bound for $\sum_{i=m}^{M}b((B_{i}^{k})^{-1}r_{i}^{k},r_{i}^{k})$ by estimating two terms
\begin{align*}
b((B_{i}^{k})^{-1}r_{i}^{k},r_{i}^{k})
&\geq |b(Q_{R}^{h}(B_{i}^{k})^{-1}Q_{R}^{h}r_{i}^{k},Q_{R}^{h}r_{i}^{k})|-
\{|b(Q_{R}^{h}(B_{i}^{k})^{-1}Q_{*}^{h}r_{i}^{k},Q_{R}^{h}r_{i}^{k})|\\
&\ \ \ \ +|b(Q_{*}^{h}(B_{i}^{k})^{-1}Q_{R}^{h}r_{i}^{k},Q_{*}^{h}r_{i}^{k})|
+|b(Q_{*}^{h}(B_{i}^{k})^{-1}Q_{*}^{h}r_{i}^{k},Q_{*}^{h}r_{i}^{k})|\}\\
&=:J_{RR}-J_{R*}.
\end{align*}
For the first term $J_{RR}$, by \eqref{align_estimate_theta_a_1M}, \eqref{align_theta_ak_def_1M} and using the same argument as \eqref{align_operator_eig}, we get
\begin{align*}
J_{RR}
&=|(Q_{R}^{h}\sum_{l=0}^{N}T_{l,i}^{k}Q_{R}^{h}u_{i}^{k},Q_{R}^{h}u_{i}^{k})_{E_{i,R}^{k}}|
\geq \frac{1}{C_{i}\left(1+\frac{H}{\delta}\right)}(Q_{R}^{h}u_{i}^{k},Q_{R}^{h}u_{i}^{k})_{E_{i,R}^{k}}\\
&\geq C||Q_{R}^{h}g_{i}^{k}||_{a}^{2}
\geq C(||g_{i}^{k}||_{a}^{2}-(\theta_{a,*}^{k})^{2}||g_{i}^{k}||_{a}^{2})
\geq C(1-CH^{2})||g_{i}^{k}||_{a}^{2}.
\end{align*}
For the second term $J_{R*}$, using Lemma \ref{Lemma_QLQJ} and the Poincar\'e inequality, since $(I-P^{k})g_{i}^{k}=g_{i}^{k}$, we know 
$J_{R*}\leq CH||g_{i}^{k}||_{a}^{2}.$ Hence, taking $\tau=\frac{1}{C_{N_{0}}}$, we obtain
\begin{align*}
\sum_{i=m}^{M}(\lambda_{i}^{k}-\lambda_{i,1})\geq \frac{1}{C_{N_{0}}} \sum_{i=m}^{M}b((B_{i}^{k})^{-1}r_{i}^{k},r_{i}^{k})\geq C\sum_{i=m}^{M}||g_{i}^{k}||_{a}^{2}.
\end{align*}
\par Next, we prove the right inequality of \eqref{Lemma_suffice_gika}. Using the same argument as Lemma 8 in \cite{Ovtchinnikov} and the fact $B_{12}=0$, we know that
\begin{align}\label{algn_check_i1}
\sum_{i=m}^{M}(\check{\lambda}_{i}^{k+1}-\lambda_{i,1})
\leq \sum_{i=m}^{M}\sum_{j=m}^{M}\frac{\check{\lambda}_{i}^{k+1}||A_{12}||^{2}}{(\check{\lambda}_{j}^{k+1}-\lambda_{m-1,0})^{2}}.
\end{align}
Noting that for any $i=m,m+1,\cdots,M$, $j=1,2,\cdots,m-1$,
\begin{align*}
a(\bar{u}_{i}^{k+1},u_{j}^{k})
=\tau a(q_{i}^{k},u_{j}^{k})
=\tau \lambda_{j}^{k} a((I-Q^{k})t_{i}^{k},g_{j}^{k})
\leq C||(B_{i}^{k})^{-1}r_{i}^{k}||_{a}||g_{j}^{k}||_{a},
\end{align*}
we deduce
\begin{align*}
||A_{12}||^{2}
&\leq ||A_{12}||_{F}^{2}
=\sum_{i=m}^{M}\sum_{j=1}^{m-1}\left(a(\bar{u}_{i}^{k+1},u_{j}^{k})\right)^{2}
\leq C\sum_{j=1}^{m-1}||g_{j}^{k}||_{a}^{2}
\sum_{i=m}^{M}||(B_{i}^{k})^{-1}r_{i}^{k}||_{a}^{2},
\end{align*}
which, together with \eqref{algn_check_i1} and Assumption \ref{Assumption1}, yields
\begin{align*}
\sum_{i=m}^{M}(\check{\lambda}_{i}^{k+1}-\lambda_{i,1})
&\leq C\sum_{j=1}^{m-1}||g_{j}^{k}||_{a}^{2}
\sum_{i=m}^{M}||(B_{i}^{k})^{-1}r_{i}^{k}||_{a}^{2}\\
&\leq C\sum_{j=1}^{m-1}(\mu_{j}^{h}-\mu_{j}^{0})
\sum_{i=m}^{M}(C||e_{i,R}^{k}||_{E_{i,R}^{k}}^{2}+CH^{2}||g_{i}^{k}||_{a}^{2}).
\end{align*}
Since $||e_{i,R}^{k}||_{E_{i,R}^{k}}^{2}
=b((A^{h}-\lambda_{i}^{k})Q_{R}^{h}u_{i}^{k},Q_{R}^{h}u_{i}^{k})
=\lambda_{i}^{k}a((\mu_{i}^{k}-T^{h})Q_{R}^{h}u_{i}^{k},Q_{R}^{h}u_{i}^{k})
\leq C||g_{i}^{k}||_{a}^{2},$
we obtain $\sum_{i=m}^{M}(\check{\lambda}_{i}^{k+1}-\lambda_{i,1})
\leq C\epsilon(H)\sum_{i=m}^{M}||g_{i}^{k}||_{a}^{2},$
where $\epsilon(H)=\sum_{j=1}^{m-1}(\mu_{j}^{h}-\mu_{j}^{0})\ (\leq CH^{2})\to 0$, as $H\to 0$.
\end{proof}

\par Now we are in a position to prove the main result in this paper.
\par\noindent{\bf Proof of Theorem \ref{Theorem_PJD_conver_rate}}:\ \
By \eqref{hat_eig_inequ}, we know $\sum_{i=m}^{M}(\lambda_{i}^{k+1}-\lambda_{i}^{h})
\leq \sum_{i=m}^{M}(\hat{\lambda}_{i}^{k+1}-\lambda_{i}^{h}).$
Then dividing the term $\sum_{i=m}^{M}(\hat{\lambda}_{i}^{k+1}-\lambda_{i}^{h})$ into three terms and using Lemma \ref{Lemma_gamma_max} and Lemma \ref{Lemma_estimate_lam_kh}, we obtain
\begin{align*}
\sum_{i=m}^{M}(\hat{\lambda}_{i}^{k+1}-\lambda_{i}^{h})
&=\sum_{i=m}^{M}(\hat{\lambda}_{i}^{k+1}-Rq(\widetilde{u}_{i}^{k+1}))
+\sum_{i=m}^{M}(Rq(\widetilde{u}_{i}^{k+1})-Rq(u_{i,J}^{k}))\\
&\ \ \ \ +\sum_{i=m}^{M}(Rq(u_{i,J}^{k})-\lambda_{i}^{h})
\leq  (\gamma_{\max}+CH)\sum_{i=m}^{M}\eta_{i}^{k}+CH^{2}\sum_{i=m}^{M}||g_{i}^{k}||_{a}^{2},
\end{align*}
where the estimate result $\sum_{i=m}^{M}(Rq(u_{i,J}^{k})-\lambda_{i}^{h})\leq CH\sum_{i=m}^{M}||g_{i}^{k}||_{a}^{2}$ holds 
by using Lemma \ref{Lemma_estimate_UkUJ} and the same technique as in the proof of Lemma 2 in \cite{MR2214752}.
Moreover, by Lemma \ref{Lemma_estimate_gik} and \eqref{check_eig_inequ}, we deduce
\begin{align*}
\sum_{i=m}^{M}(\lambda_{i}^{k+1}-\lambda_{i}^{h})
&\leq (\gamma_{\max}+CH)\sum_{i=m}^{M}\eta_{i}^{k}+CH\sum_{i=m}^{M}||g_{i}^{k}||_{a}^{2}\\
&\leq \widetilde{\gamma}_{\max}\sum_{i=m}^{M}(\lambda_{i}^{k}-\lambda_{i}^{h})+CH\sum_{i=m}^{M}(\lambda_{i}^{k}-\check{\lambda}_{i}^{k+1})\\
&\leq (\widetilde{\gamma}_{\max}+CH)\sum_{i=m}^{M}(\lambda_{i}^{k}-\lambda_{i}^{h})+CH\sum_{i=m}^{M}(\lambda_{i}^{h}-\lambda_{i}^{k+1}),
\end{align*}
where $\widetilde{\gamma}_{\max}=\gamma_{\max}+CH<1$ for sufficiently small $H$. Therefore, $$\sum_{i=m}^{M}(\lambda_{i}^{k+1}-\lambda_{i}^{h})\leq \frac{\widetilde{\gamma}_{\max}+CH}{1+CH}\sum_{i=m}^{M}(\lambda_{i}^{k}-\lambda_{i}^{h})=: c(H)\rho(\frac{\delta}{H},d_{m}^{-},d_{M}^{+})\sum_{i=m}^{M}(\lambda_{i}^{k}-\lambda_{i}^{h}),$$
which means that \eqref{Equation_Theorem eigenvalues} holds. Combining  \eqref{align_estimate_theta_a} and \eqref{Equation_Theorem eigenvalues}, we obtain \eqref{Equation_Theorem eigenvector_a}. \hfill \proofbox

\section{ Numerical Experiments}\label{sec5}
\par In this section, we present several numerical experiments to support our theoretical findings. For the stopping criterion, we choose the accuracy of
$ stop.=\sqrt{\sum_{i=m}^{M}||r_{i}^{k}||^{2}_{b}}<tol=10^{-8}.$
\begin{example}
We consider to use the $P_{1}$-conforming finite element on uniform mesh to approximate the Laplacian eigenvalue problems in $\Omega=(0,\pi)^{2}$ with homogeneous Dirichlet boundary condition on $\partial\Omega$. Set the coarse mesh $H=\frac{\sqrt{2}\pi}{2^{5}}$. Now, we fix the ratio $\frac{\delta}{H}=\frac{1}{4}$ and test optimality and scalability the two-level additive Schwarz method.
\end{example}

\par Firstly, we compute the interior multiple eigenvalues from the 99th to the 108th: $\lambda_{99}=\lambda_{100}=\lambda_{101}=\lambda_{102}=145$, $\lambda_{103}=\lambda_{104}=146$, $\lambda_{105}=\lambda_{106}=148$ and
$\lambda_{107}=\lambda_{108}=149$. The initial eigenvalue approximations of the two-level additive Schwarz method on the initial mesh $\mathcal{T}_{\widetilde{H}}$ with $\widetilde{H}=\frac{\sqrt{2}\pi}{2^{5}}$ and $d.o.f.=961$ are:
$\lambda_{99}^{\widetilde{H}} =161.042017, $    $\lambda_{100}^{\widetilde{H}}=162.829383,$ 
$\lambda_{101}^{\widetilde{H}}=162.842389,$ $\lambda_{102}^{\widetilde{H}}=166.819972,$ 
$\lambda_{103}^{\widetilde{H}}=166.865099,$    $\lambda_{104}^{\widetilde{H}}=167.540267,$ 
$\lambda_{105}^{\widetilde{H}}=169.887856,$ $\lambda_{106}^{\widetilde{H}}=170.700726,$ 
$\lambda_{107}^{\widetilde{H}}=173.459262,$ 
$\lambda_{108}^{\widetilde{H}}=174.709492.$ 
\begin{table}[H]\small
\centering
\caption{ $H=\frac{\sqrt{2}\pi}{2^{5}}$, $\frac{\delta}{H}=\frac{1}{4}$, $m=99,\ M=108$}\label{Table_ConvexOptimal1}
\newcolumntype{d}{D{.}{.}{2}}
\begin{tabular}{|c|c|c|c|c|c|c|}
\hline
\multicolumn{1}{|c|}{$h=\frac{\sqrt{2}\pi}{2^{j}}$}&\multicolumn{1}{c|}{$j=8$}
&\multicolumn{1}{c|}{$j=9$}&\multicolumn{1}{c|}{$j=10$}
&\multicolumn{1}{c|}{$j=11$}&\multicolumn{1}{c|}{$j=12$}\\
\hline
\multicolumn{1}{|c|}{$d.o.f.$}&\multicolumn{1}{c|}{$65025$}
&\multicolumn{1}{c|}{$261121$}&\multicolumn{1}{c|}{$1046529$}
&\multicolumn{1}{c|}{$4190209$}&\multicolumn{1}{c|}{$16769025$}\\
\hline
\multicolumn{1}{|c|}{$\lambda_{i}$}&\multicolumn{1}{c|}{$36(it.)$}
&\multicolumn{1}{c|}{$35(it.)$}&\multicolumn{1}{c|}{$34(it.)$}
&\multicolumn{1}{c|}{$32(it.)$}&\multicolumn{1}{c|}{$34(it.)$}\\
\hline
$\lambda_{99}=145$      &145.267540  &145.066850    &145.016710  &145.004177   &145.001076  \\\hline
$\lambda_{100}=145$     &145.267710  &145.066890    &145.016720  &145.004180   &145.001179  \\\hline
$\lambda_{101}=145$     &145.285465  &145.071880    &145.018002  &145.004502   &145.001467  \\\hline
$\lambda_{102}=145$     &145.497479  &145.124812    &145.031230  &145.007809   &145.001968  \\\hline
$\lambda_{103}=146$     &146.343640  &146.085872    &146.021466  &146.005366   &146.001386  \\\hline
$\lambda_{104}=146$     &146.343692  &146.085876    &146.021466  &146.005366   &146.001387  \\\hline
$\lambda_{105}=148$     &148.289692  &148.072359    &148.018086  &148.004521   &148.001158  \\\hline
$\lambda_{106}=148$     &148.289716  &148.072360    &148.018086  &148.004521   &148.001160  \\\hline
$\lambda_{107}=149$     &149.391777  &149.097649    &149.024393  &149.006097   &149.001578  \\\hline
$\lambda_{108}=149$     &149.414814  &149.103366    &149.025821  &149.006454   &149.001585  \\\hline
$stop.$                 &8.6089e-09    &9.6646e-09      &7.2879e-09    &9.8316e-09     &5.4729e-09    \\\hline
\end{tabular}
\end{table}
\begin{table}[H]\small
 \centering
 \caption{$H=\frac{\sqrt{2}\pi}{2^{5}},\frac{\sqrt{2}\pi}{2^{6}},\frac{\sqrt{2}\pi}{2^{7}}$, $\frac{\delta}{H}=\frac{1}{4}$, $m=99,\ M=108$}\label{Table_ConvexScalable1}
\newcolumntype{d}{D{.}{.}{2}}
\begin{tabular}{|c|c|c|c|}
\hline
$H$&\multicolumn{1}{c|}{$d.o.f.$}&\multicolumn{1}{c|}{$stop.$}&\multicolumn{1}{c|}{$it.$}\\
\hline
$\frac{\sqrt{2}\pi}{2^{5}}$&16769025&5.4729e-09&34\\
\hline
$\frac{\sqrt{2}\pi}{2^{6}}$&16769025&3.7823e-09&29\\
\hline
$\frac{\sqrt{2}\pi}{2^{7}}$&16769025&7.2187e-09&26\\
\hline
\end{tabular}
\end{table}
\par Before we analyze the numerical results, we introduce some notations: $d.o.f.$ denotes the degree of freedoms of $V^{h}$, and $it.$ represents the number of iteration of the proposed method. Numerical results in Table  \ref{Table_ConvexOptimal1} in each row from $\lambda_{99}$ to $\lambda_{108}$ shows that the iterative approximations $\lambda_{i}^{it.}$ of targeted eigenvalues converge to the discrete eigenvalues $\lambda_{i}^{h},\ i=99,100,...,108$, respectively.  From table \ref{Table_ConvexOptimal1}, the number of iteration keeps nearly stable when $d.o.f.\ \to \infty$, i.e., $h\to 0$, which illustrates that the convergence factor of the two-level additive Schwarz method is optimal. It is shown in table \ref{Table_ConvexScalable1} that the number of iteration does not increase when $H\ \to 0$, which shows that the two-level additive Schwarz method is scalable.

\par In order to present that our two-level additive Schwarz method is optimal and scalable to solve more interior eigenvalues with corresponding eigenfunctions, we compute the interior multiple eigenvalues from the 198th to the 203th: $\lambda_{198}=\lambda_{199}=272$, $\lambda_{200}=\lambda_{201}=274$ and $\lambda_{202}=\lambda_{203}=277$. The initial eigenvalue approximations of the two-level additive Schwarz method on the initial mesh $\mathcal{T}_{\widetilde{H}}$ with $\widetilde{H}=\frac{\sqrt{2}\pi}{2^{5}}$ and $d.o.f.=961$ are:
$\lambda_{198}^{\widetilde{H}}=342.963070,$ $\lambda_{199}^{\widetilde{H}}=351.359467,$
$\lambda_{200}^{\widetilde{H}}=352.358152,$ $\lambda_{201}^{\widetilde{H}}=358.045139,$
$\lambda_{202}^{\widetilde{H}}=358.442449,$ $\lambda_{203}^{\widetilde{H}}=359.752727.$

\begin{table}[H]\small
 \centering
 \caption{ $H=\frac{\sqrt{2}\pi}{2^{5}}$, $\frac{\delta}{H}=\frac{1}{4}$, $m=198,\ M=203$}\label{Table_ConvexOptimal2}
\newcolumntype{d}{D{.}{.}{2}}
\begin{tabular}{|c|c|c|c|c|c|c|}
\hline
\multicolumn{1}{|c|}{$h=\frac{\sqrt{2}\pi}{2^{j}}$}&\multicolumn{1}{c|}{$j=8$}
&\multicolumn{1}{c|}{$j=9$}&\multicolumn{1}{c|}{$j=10$}
&\multicolumn{1}{c|}{$j=11$}&\multicolumn{1}{c|}{$j=12$}\\
\hline
\multicolumn{1}{|c|}{$d.o.f.$}&\multicolumn{1}{c|}{$65025$}
&\multicolumn{1}{c|}{$261121$}&\multicolumn{1}{c|}{$1046529$}
&\multicolumn{1}{c|}{$4190209$}&\multicolumn{1}{c|}{$16769025$}\\
\hline
\multicolumn{1}{|c|}{$\lambda_{i}$}&\multicolumn{1}{c|}{$64(it.)$}
&\multicolumn{1}{c|}{$64(it.)$}&\multicolumn{1}{c|}{$63(it.)$}
&\multicolumn{1}{c|}{$61(it.)$}&\multicolumn{1}{c|}{$62(it.)$}\\
\hline
$\lambda_{198}=272$     &273.030189    &272.257745   &272.064451  &272.016114   &272.004137  \\\hline
$\lambda_{199}=272$     &273.030378    &272.257757   &272.064452  &272.016114   &272.004298  \\\hline
$\lambda_{200}=274$     &275.223475    &274.305023   &274.076201  &274.019047   &274.004743  \\\hline
$\lambda_{201}=274$     &275.223917    &274.305051   &274.076203  &274.019048   &274.004748  \\\hline
$\lambda_{202}=277$     &278.354888    &277.337414   &277.084271  &277.021063   &277.005288  \\\hline
$\lambda_{203}=277$     &278.382358    &277.344192   &277.085963  &277.021486   &277.005298  \\\hline
$stop.$                 &9.3632e-09    &7.5421e-09    &8.0610e-09    &7.2423e-09     &9.0264e-09    \\\hline
\end{tabular}
\end{table}

\begin{table}[H]\small
 \centering
 \caption{$H=\frac{\sqrt{2}\pi}{2^{5}},\frac{\sqrt{2}\pi}{2^{6}},\frac{\sqrt{2}\pi}{2^{7}}$, $\frac{\delta}{H}=\frac{1}{4}$, $m=198,\ M=203$}\label{Table_ConvexScalable2}
\newcolumntype{d}{D{.}{.}{2}}
\begin{tabular}{|c|c|c|c|}
\hline
$H$&\multicolumn{1}{c|}{$d.o.f.$}&\multicolumn{1}{c|}{$stop.$}&\multicolumn{1}{c|}{$it.$}\\
\hline
$\frac{\sqrt{2}\pi}{2^{5}}$&16769025&9.0264e-09&62\\
\hline
$\frac{\sqrt{2}\pi}{2^{6}}$&16769025&4.3782e-09&55\\
\hline
$\frac{\sqrt{2}\pi}{2^{7}}$&16769025&6.3921e-09&49\\
\hline
\end{tabular}
\end{table}

\par Similarly, from tables \ref{Table_ConvexOptimal2} and \ref{Table_ConvexScalable2}, the two-level additive Schwarz method for computing interior eigenvalues $\{\lambda_{i}\}_{i=198}^{203}$ is optimal and scalable. Finally, we show another group of numerical results for the Laplacian eigenvalue problems. We compute the interior multiple eigenvalues from the 499th to the 502th: $\lambda_{499}=\lambda_{500}=673$ and $\lambda_{501}=\lambda_{502}=674$. The initial eigenvalue approximations of the proposed method on the initial mesh $\mathcal{T}_{\widetilde{H}}$ with $\widetilde{H}=\frac{\sqrt{2}\pi}{2^{6}}$ and $d.o.f.=3969$ are
$\lambda_{499}^{\widetilde{H}}=778.376460,$
$\lambda_{500}^{\widetilde{H}}=780.027544,$
$\lambda_{501}^{\widetilde{H}}=780.041908,$  $\lambda_{502}^{\widetilde{H}}=780.595760.$

\begin{table}[H]\small
 \centering
 \caption{ $H=\frac{\sqrt{2}\pi}{2^{6}}$, $\frac{\delta}{H}=\frac{1}{4}$, $m=499,\ M=502$}\label{Table_ConvexOptimal3}
\newcolumntype{d}{D{.}{.}{2}}
\begin{tabular}{|c|c|c|c|c|c|}
\hline
\multicolumn{1}{|c|}{$h=\frac{\sqrt{2}\pi}{2^{j}}$}&\multicolumn{1}{c|}{$j=8$}&\multicolumn{1}{c|}{$j=9$}
&\multicolumn{1}{c|}{$j=10$}&\multicolumn{1}{c|}{$j=11$}&\multicolumn{1}{c|}{$j=12$}\\
\hline
\multicolumn{1}{|c|}{$d.o.f.$}&\multicolumn{1}{c|}{$65025$}&\multicolumn{1}{c|}{$261121$}
&\multicolumn{1}{c|}{$1046529$}&\multicolumn{1}{c|}{$4190209$}&\multicolumn{1}{c|}{$16769025$}\\
\hline
\multicolumn{1}{|c|}{$\lambda_{i}$}
&\multicolumn{1}{c|}{$90(it.)$}&\multicolumn{1}{c|}{$96(it.)$}&\multicolumn{1}{c|}{$93(it.)$}
&\multicolumn{1}{c|}{$93(it.)$}&\multicolumn{1}{c|}{$92(it.)$}\\
\hline
$\lambda_{499}=673$     &680.405341  &674.899653  &673.475776  &673.118289   &673.029499  \\\hline
$\lambda_{500}=673$     &680.407954  &674.906418  &673.477268  &673.118343   &673.029660  \\\hline
$\lambda_{501}=674$     &680.639711  &675.612138  &674.404020  &674.101373   &674.025125  \\\hline
$\lambda_{502}=674$     &680.669120  &675.612340  &674.404227  &674.101540   &674.025449  \\\hline
$stop.$                 &9.9989e-09    &7.5421e-09    &8.0069e-09    &4.2351e-09     &2.2291e-09 \\\hline
\end{tabular}
\end{table}

\begin{table}[H]\small
 \centering
 \caption{$H=\frac{\sqrt{2}\pi}{2^{6}},\frac{\sqrt{2}\pi}{2^{7}}$, $\frac{\delta}{H}=\frac{1}{4}$, $m=499,\ M=502$}\label{Table_ConvexScalable3}
\newcolumntype{d}{D{.}{.}{2}}
\begin{tabular}{|c|c|c|c|}
\hline
$H$&\multicolumn{1}{c|}{$d.o.f.$}&\multicolumn{1}{c|}{$stop.$}&\multicolumn{1}{c|}{$it.$}\\
\hline
$\frac{\sqrt{2}\pi}{2^{6}}$&16769025&2.2291e-09&92\\
\hline
$\frac{\sqrt{2}\pi}{2^{7}}$&16769025&3.7644e-09&80\\
\hline
\end{tabular}
\end{table}

\par Although we only give the analysis for convex case, our method works well for nonconvex case. When the computed domain is not so regular, some interior eigenvalues of the differential operators may form a eigenvalue cluster. Now, we consider to use the two-level additive Schwarz method to solve interior multiple and clustered eigenvalues of the Laplacian operator in the L-shaped domain.
\begin{example}
We consider to use the $P_{1}$-conforming finite element on uniform mesh to approximate the Laplacian eigenvalue problems in L-shaped domain $(-\pi,\pi)^{2}$ $\backslash[0,\pi)\times(-\pi,0]$ with homogeneous Dirichlet boundary condition on $\partial\Omega$. Set the coarse mesh $H=\frac{\sqrt{2}\pi}{2^{4}}$. Now, we fix the ratio $\frac{\delta}{H}=\frac{1}{4}$ and test optimality and scalability the two-level additive Schwarz method.
\end{example}

\par Firstly, we compute the interior multiple and clustered eigenvalues from the 41th to the 47th which are unknown. The initial eigenvalue approximations of the proposed method on the initial mesh $\mathcal{T}_{\widetilde{H}}$ with $\widetilde{H}=\frac{\sqrt{2}\pi}{2^{4}}$ and $d.o.f.=705$ are: $\lambda_{41}^{\widetilde{H}}=22.729142,$
$\lambda_{42}^{\widetilde{H}}=22.764697,$
$\lambda_{43}^{\widetilde{H}}=23.006578,$ $\lambda_{44}^{\widetilde{H}}=25.000928,$
$\lambda_{45}^{\widetilde{H}}=25.581274,$ $\lambda_{46}^{\widetilde{H}}=25.984862,$  $\lambda_{47}^{\widetilde{H}}=26.245603.$

\begin{table}[H]\small
 \centering
 \caption{ $H=\frac{\sqrt{2}\pi}{2^{4}}$, $\frac{\delta}{H}=\frac{1}{4}$, $m=41,\ M=47$}\label{Table_nonConvexOptimal_v1}
\newcolumntype{d}{D{.}{.}{2}}
\begin{tabular}{|c|c|c|c|c|c|c|}
\hline
\multicolumn{1}{|c|}{$h=\frac{\sqrt{2}\pi}{2^{j}}$}&\multicolumn{1}{c|}{$j=7$}
&\multicolumn{1}{c|}{$j=8$}&\multicolumn{1}{c|}{$j=9$}
&\multicolumn{1}{c|}{$j=10$}&\multicolumn{1}{c|}{$j=11$}\\
\hline
\multicolumn{1}{|c|}{$d.o.f.$}&\multicolumn{1}{c|}{$48641$}
&\multicolumn{1}{c|}{$195585$}&\multicolumn{1}{c|}{$784385$}
&\multicolumn{1}{c|}{$3141633$}&\multicolumn{1}{c|}{$12574721$}\\
\hline
\multicolumn{1}{|c|}{$\lambda_{i}$}&\multicolumn{1}{c|}{14$(it.)$}
&\multicolumn{1}{c|}{12 $(it.)$}&\multicolumn{1}{c|}{$11(it.)$}
&\multicolumn{1}{c|}{11$(it.)$}&\multicolumn{1}{c|}{11$(it.)$}\\
\hline
$\lambda_{41}$    &21.038261  &21.017903 &21.012693  &21.011380   &21.011047  \\\hline
$\lambda_{42}$    &21.176737  &21.159106 &21.154753  &21.153663   &21.153386  \\\hline
$\lambda_{43}$    &21.247400  &21.227294 &21.222114  &21.220853   &21.220541  \\\hline
$\lambda_{44}$    &22.654591  &22.626330 &22.619536  &22.617922   &22.617516  \\\hline
$\lambda_{45}$    &22.702806  &22.669829 &22.661640  &22.659588   &22.659088  \\\hline
$\lambda_{46}$    &23.959245  &23.932326 &23.925872  &23.924305   &23.923886  \\\hline
$\lambda_{47}$    &24.191689  &24.166038 &24.160099  &24.158572   &24.158197  \\\hline
$stop.$           &6.1398e-09   &7.0221e-09  &8.0057e-09   &5.0695e-09    &4.5618e-09   \\\hline
\end{tabular}
\end{table}

\begin{table}[H]\small
 \centering
 \caption{$H=\frac{\sqrt{2}\pi}{2^{4}},\frac{\sqrt{2}\pi}{2^{5}},\frac{\sqrt{2}\pi}{2^{6}}$, $\frac{\delta}{H}=\frac{1}{4}$, $m=41,\ M=47$}\label{Table_nonConvexScalable_v1}
\newcolumntype{d}{D{.}{.}{2}}
\begin{tabular}{|c|c|c|c|}
\hline
$H$&\multicolumn{1}{c|}{$d.o.f.$}&\multicolumn{1}{c|}{$stop.$}&\multicolumn{1}{c|}{$it.$}\\
\hline
$\frac{\sqrt{2}\pi}{2^{4}}$&12574721&4.5618e-09&11\\
\hline
$\frac{\sqrt{2}\pi}{2^{5}}$&12574721&2.8936e-09&9\\
\hline
$\frac{\sqrt{2}\pi}{2^{6}}$&12574721&7.3652e-09&7\\
\hline
\end{tabular}
\end{table}
\par Although some eigenfunctions have strong singularities in L-shaped domain, the proposed method works very well. From table \ref{Table_nonConvexOptimal_v1}, the number of iteration keeps stable when $h\to 0$, which shows that the convergence factor of the two-level additive Schwarz method is optimal. Moreover, the eigenvalues $\{\lambda_{i}\}_{i=41}^{47}$ are close to each other, but our method still works well, which illustrates that the convergence factor is not negatively impacted by gaps among the targeted clustered eigenvalues $\{\lambda_{i}\}_{i=41}^{47}$. Furthermore, it is seen in table \ref{Table_nonConvexScalable_v1} that the number of iterations does not increase with the number of subdomains increasing, which shows that the two-level additive Schwarz method is scalable. Next, we present another group results. We compute the interior multiple and clustered eigenvalues from the 87th to the 91th which are unknown. The initial eigenvalue approximations of the proposed method on the initial mesh $\mathcal{T}_{\widetilde{H}}$ with $\widetilde{H}=\frac{\sqrt{2}\pi}{2^{4}}$ and $d.o.f.=705$ are:
$\lambda_{87}^{\widetilde{H}}=49.512180,$ $\lambda_{88}^{\widetilde{H}}=49.645803,$
$\lambda_{89}^{\widetilde{H}}=50.985804,$ $\lambda_{90}^{\widetilde{H}}=51.112112,$
$\lambda_{91}^{\widetilde{H}}=51.484953.$

\begin{table}[H]\small
 \centering
 \caption{ $H=\frac{\sqrt{2}\pi}{2^{4}}$, $\frac{\delta}{H}=\frac{1}{4}$, $m=87,\ M=91$}\label{Table_nonConvexOptimal_v2}
\newcolumntype{d}{D{.}{.}{2}}
\begin{tabular}{|c|c|c|c|c|c|c|}
\hline
\multicolumn{1}{|c|}{$h=\frac{\sqrt{2}\pi}{2^{j}}$}&\multicolumn{1}{c|}{$j=7$}
&\multicolumn{1}{c|}{$j=8$}&\multicolumn{1}{c|}{$j=9$}
&\multicolumn{1}{c|}{$j=10$}&\multicolumn{1}{c|}{$j=11$}\\
\hline
\multicolumn{1}{|c|}{$d.o.f.$}&\multicolumn{1}{c|}{$48641$}
&\multicolumn{1}{c|}{$195585$}&\multicolumn{1}{c|}{$784385$}
&\multicolumn{1}{c|}{$3141633$}&\multicolumn{1}{c|}{$12574721$}\\
\hline
\multicolumn{1}{|c|}{$\lambda_{i}$}&\multicolumn{1}{c|}{24$(it.)$}
&\multicolumn{1}{c|}{22$(it.)$}&\multicolumn{1}{c|}{21$(it.)$}
&\multicolumn{1}{c|}{21$(it.)$}&\multicolumn{1}{c|}{21$(it.)$}\\
\hline
$\lambda_{87}$     &42.237428    &42.145053    &42.122122    &42.116513     &42.115206  \\\hline
$\lambda_{88}$     &42.644708    &42.550598    &42.528282    &42.523026     &42.521780  \\\hline
$\lambda_{89}$     &43.332635    &43.258123    &43.239463    &43.235030     &43.236199  \\\hline
$\lambda_{90}$     &43.378541    &43.299305    &43.280715    &43.276644     &43.275681  \\\hline
$\lambda_{91}$     &44.000091    &43.910995    &43.890111    &43.885182     &43.883964  \\\hline
$stop.$            &5.2694e-09   &3.9657e-09   &5.2652e-09   &5.0951e-09    &9.8916e-09   \\\hline
\end{tabular}
\end{table}

\begin{table}[H]\small
 \centering
 \caption{$H=\frac{\sqrt{2}\pi}{2^{4}},\frac{\sqrt{2}\pi}{2^{5}},\frac{\sqrt{2}\pi}{2^{6}}$, $\frac{\delta}{H}=\frac{1}{4}$, $m=87,\ M=91$}\label{Table_nonConvexScalable_v2}
\newcolumntype{d}{D{.}{.}{2}}
\begin{tabular}{|c|c|c|c|}
\hline
$H$&\multicolumn{1}{c|}{$d.o.f.$}&\multicolumn{1}{c|}{$stop.$}&\multicolumn{1}{c|}{$it.$}\\
\hline
$\frac{\sqrt{2}\pi}{2^{4}}$&12574721&9.8916e-09&21\\
\hline
$\frac{\sqrt{2}\pi}{2^{5}}$&12574721&4.3729e-09&18\\
\hline
$\frac{\sqrt{2}\pi}{2^{6}}$&12574721&8.4326e-09&16\\
\hline
\end{tabular}
\end{table}

\section{Conclusion}\label{sec6}

\par In this paper, based on the overlapping domain decomposition, we propose an two-level additive Schwarz method to compute interior multiple and clustered eigenvalues of symmetric elliptic operators. It may compute different eigenvalue clusters simultaneously and compute more interior multiple and clustered eigenvalues. The theoretical analysis reveals that the two-level additive Schwarz method is optimal, scalable and cluster robust, i.e., the convergence factor of the proposed algorithm is independent of the fine mesh size, the number of subdomains and the gaps among the interior multiple and clustered eigenvalues.

\begin{small}
\bibliographystyle{siamplain}
\bibliography{reference}
\end{small}
\end{document}